\documentclass[11pt]{article}
\pdfoutput=1
\usepackage{url}
\usepackage[margin=1.3in]{geometry}
\usepackage{epsfig,cite,multirow}
\usepackage{graphicx,setspace}
\usepackage{algorithm}
\usepackage{algorithmic}
\usepackage{relsize}

\usepackage{amsmath,amssymb,amsfonts}
 \usepackage{amsthm}
\usepackage{epsfig}
\usepackage{amsxtra}
\usepackage{amsbsy}
\usepackage{amsopn}
\usepackage{epsf}
\usepackage{subfigure}
\usepackage{psfrag}
\usepackage{color}
\usepackage[usenames,dvipsnames]{xcolor}
\usepackage{tikz} 
\usepackage{currfile}
\usepackage{filehook}
\usepackage{tablefootnote}

\def\d{\delta}
\def\r{\rho}

\def\om{\omega}
\def\a{\alpha}
\def\b{\beta}

\def\gm{\gamma}
\def\Gm{\Gamma}
\def\A{\mathbf{A}}
\def\g{\mathbf{g}}
\def\P{\mathbf{P}}

\def\u{\mathbf{u}}

\def\RR{\mathbb{R}}

\def\E{\mathbb{E}}
\def\P{\mathbb{P}}
\def\ZZ{\mathcal{Z}}
\def\X{\mathcal{X}}
\def\VV{\mathcal{V}}

\newcommand{\support}{\mathcal S}
\newcommand{\suppest}{\widetilde{\mathcal S}}

\newcommand{\obs}{\mathbf y}
\newcommand{\x}{\mathbf x}
\newcommand{\z}{\mathbf z}
\newcommand{\sol}{\widehat{\x}}
\newcommand{\w}{\boldsymbol{\omega}}

\newcommand{\noise}{\mathbf e}

\newcommand{\NN}{\mathcal N}
\newcommand{\bigO}{\mathcal O}

\newcommand*{\bfrac}[2]{\genfrac{}{}{0pt}{}{#1}{#2}}

\newtheorem{theorem}{Theorem}[section]

\newtheorem{proposition}[theorem]{Proposition}
\newtheorem{definition}[theorem]{Definition}
\newtheorem{lemma}[theorem]{Lemma}

\title{The sample complexity of weighted sparse approximation}

\author{Bubacarr Bah, and Rachel Ward \thanks{Mathematics Department, and Institute of Computational Engineering and Sciences (ICES), University of Texas at Austin ({\tt bah@math.utexas.edu, rward@math.utexas.edu}).}}

\begin{document}

\maketitle

\pagestyle{myheadings}
\thispagestyle{empty}
\markboth{BUBACARR BAH, AND RACHEL WARD}{The sample complexity of weighted sparse approximation}

\pagestyle{myheadings}

\vspace{-2mm}\begin{abstract}
For Gaussian sampling matrices, we provide bounds on the minimal number of measurements $m$ required to achieve robust weighted sparse recovery guarantees in terms of how well a given prior model for the sparsity support aligns with the true underlying support.  Our main contribution is that for a sparse vector ${\bf x} \in \mathbb{R}^N$ supported on an unknown set $\mathcal{S} \subset \{1, \dots, N\}$ with $|\mathcal{S}|\leq k$, if $\mathcal{S}$ has \emph{weighted cardinality} $\omega(\mathcal{S}) := \sum_{j \in \mathcal{S}} \omega_j^2$, and if the weights on $\mathcal{S}^c$ exhibit mild growth, $\mathop{\omega_j^2 \geq  \gamma \log(j/\omega(\mathcal{S}))}$ for $j\in\mathcal{S}^c$ and $\gamma > 0$, then the sample complexity for sparse recovery via weighted $\ell_1$-minimization using weights $\omega_j$ is linear in the weighted sparsity level, $\mathop{m = \mathcal{O}(\omega(\mathcal{S})/\gamma)}$. This main result is a generalization of special cases including a) the standard sparse recovery setting where all weights $\omega_j \equiv 1$, and $m = \mathcal{O}\left(k\log\left(N/k\right)\right)$; b) the setting where the support is known a priori, and $m = \mathcal{O}(k)$; and c) the setting of sparse recovery with prior information, and $m$ depends on how well the weights are aligned with the support set $\mathcal{S}$. We further extend the results in case c) to the setting of additive noise. Our results are {\em nonuniform} that is they apply for a fixed support, unknown a priori, and the weights on $\mathcal{S}$ do not all have to be smaller than the weights on $\mathcal{S}^c$ for our recovery results to hold.
\end{abstract}


\section{Introduction} \label{sec:intro}

Consider the scenario where we would like to acquire an unknown vector $\x \in \mathbb{R}^N$ with assumed sparse or compressible representation from $m \ll N$ linear measurements of the form $y_i = \langle {\bf a}_i, \x \rangle$.   We are also interested in the noisy setting where instead $y_i = \langle {\bf a}_i, \x \rangle + e_i$ where $e_i$ is the $i^{\mbox{th}}$ component of a noise vector $\noise$ with noise level $\| \noise \|_2 \leq \eta$.   Recall that $\x \in \mathbb{R}^N$ is said to be $k$-sparse if it is supported on a subset $\support \subset \{1,2, \dots, N\}$ of cardinality $| \support | \leq k$.   It is well known by now that for an unknown $k$-sparse vector $\x$, given a measurement matrix $\A \in \RR^{m \times N}$, $m = \bigO( k \log(N/k))$ noisy random linear measurements $\obs = \A \x + \noise$ suffice to guarantee stable and robust reconstruction of $\x$ via the following $\ell_1$-minimization problem:
\begin{equation}
\label{eqn:l1min}\tag{L1}
\x^{\#} := \text{arg}\min_{\z\in\RR^N} ~\|\z\|_{1} \quad \mbox{subject to} \quad \| \A\z - \obs \|_2 \leq \eta\,,\vspace{-1mm}
\end{equation}
with the solution having an approximation error of the form $\| \x^{\#} - \x \|_2 \leq C_1 \| \x - \x_k \|_2 + C_2 \eta$, where $\x_k$ is the best $k$-term approximation to $\x$.  In fact, the number of measurements $\mathop{m \geq C k \log(N/k)}$ is order-optimal -- given no additional information about the support of $\x$, no algorithm can provide a better worst-case recovery guarantee using fewer measurements.  This forms the basis of the field of compressive sensing  \cite{donoho2006compressed,candes2006stable,foucart2013mathematical}.

In many practical problems of interest, one might have additional prior information about which indices within the full index set $[N] = \{1,2, \dots, N\}$ are more likely to appear in the best $k$-term approximation.   So in the noisy setting above, given a prior distribution of weights $\om_j > 0$ associated to the indices, a natural choice of reconstruction algorithm is weighted $\ell_1$-minimization:
\begin{equation}
\label{eqn:wl1min}\tag{WL1}
\sol := \text{arg}\min_{\z\in\RR^N} ~\|\z\|_{\om,1} \quad \mbox{subject to} \quad \|\A\z - \obs\|_2 \leq \eta\,,\vspace{-1mm}
\end{equation}
where $\|\z\|_{\om,1} = \sum_{j=1}^N \om_j|z_j|$ with larger-magnitude weights (or, perhaps more precisely,  larger-magnitude ``penalties") $\om_j > 0$ are assigned to indices believed less likely to contribute to the support of $\x$.  In case the prior distribution of weights is in fact aligned with the sparsity structure of $\x$, it is natural to ask whether weighted $\ell_1$-minimization might have reduced sample complexity compared to unweighted $\ell_1$-minimization, \eqref{eqn:l1min}.  This has been the subject of several works in the area of compressive sensing, see for example \cite{von2007compressed,candes2008enhancing,khajehnejad2009weighted,vaswani2010modified,jacques2010short,xu2010breaking,mansour2011weighted,friedlander2012recovering,oymak2012recovery,mansour2012support,rauhut2015interpolation,mansour2014recovery}. 

Existing results show better sparse reconstruction results for weighted $\ell_1$-minimization than the standard unweighted $\ell_1$-minimization, which is usually demonstrated in the form of weaker recovery conditions that guarantee recovery with reduced sample complexities \cite{khajehnejad2009weighted,rauhut2015interpolation,mansour2014recovery}.  
 In the extreme case, one might know the support of $\x$ exactly; say, for example, $\support = \{1,2, \dots, k\}$.   In the noiseless setting, $\x$ can then be exactly recovered from $\obs$ using e.g. $m = k$ Gaussian random measurements, with probability $1$, the measurements $\obs = \A \x = \A_{\support} \x_{\support}$ will be invertible. 
In the middle ground between knowing the support set a priori and knowing only that the support is some subset of $[N]$ of cardinality $k$, what is the sample complexity of sparse recovery?   Clearly the answer lies between the extremes of $m = \bigO(k)$ and $m = \bigO( k \log(N/k))$.  The purpose of this paper is to make the sample complexity precise, in terms of the distribution of the weights and how well this distribution aligns with the true underlying support of the signals of interest.

There are several motivating examples in which one naturally is faced with a sparse recovery problem and a prior distribution over the support.
For example, the smoothness of a function corresponds to the rapid decay of its Fourier series, so lower-order frequency basis functions are more likely to contribute to the best $s$-term approximation, hence making smoothness a structured sparsity constraint. It has been demonstrated in \cite{rauhut2015interpolation} that such structured sparse expansions are imposed by weighted $\ell_p$ coefficient spaces in the range $0<p\leq 1$, and that weighted $\ell_1$-minimization can be used as a convex surrogate for weighted $\ell_p$-minimization with $p<1$.

Alternatively, the recursive recovery of time sequences of sparse signals when their supports evolves slowly over time \cite{jacques2010short} is another setting with prior information about the support. In this case the support from the previous time can be used as a support estimate in the current time. This can be applied to real-time dynamic MRI reconstruction, real-time single-pixel camera video imaging.
More  applications  where sparse  reconstruction  for  time  sequences  of  signals/images may be needed can be found in\cite{vaswani2010modified}. 

Similarly in medical imaging, natural images have tree-structured wavelet expansions and are less sparse on lower-order wavelets. Precisely, wavelet functions act as local discontinuity detectors, and a signal discontinuity will give rise to a chain of large wavelet coefficients along a branch of the wavelet tree from a leaf to the root. Leading to tree-sparsity where indices in a support form a rooted-connected tree. For such wavelet-tree sparsity it is possible to make choices of weights in accordance with the sparsity distribution. This motivated the development of the notion of {\em weighted sparsity} in \cite{rauhut2015interpolation} which takes into account prior information on the likelihood that any particular index is part of the wavelet tree-sparse support. 
 
With regards to related works, to the best of our knowledge, performance guarantees for weighted $\ell_1$ minimization in the compressive sensing set-up were independently introduced by the works of von Borries et al. \cite{von2007compressed}, Vaswani and Lu \cite{vaswani2010modified}, and Khajehnejad et al. \cite{khajehnejad2009weighted}.
The authors of \cite{von2007compressed} showed that exploiting available support information leads to reduced sampling rates for signals with a sparse discrete Fourier transform (DFT); while the authors of \cite{vaswani2010modified} also used weighted $\ell_1$-minimization with zero weights on the support estimate and derived, weaker than $\ell_1$-minimization, sufficient recovery conditions where a large proportion of the support is known. 
This work was extended to the noisy setting
 and to compressible signals by \cite{jacques2010short}.
Similarly the authors of \cite{khajehnejad2009weighted} proposed the weighted $\ell_1$-minimization for the sparse recovery problem with known support and analyzed the performance of this method using a Grassmann angle approach. 
An improvement of their result, in the form of a closed form expression for the minimum number of measurements, was give in \cite{oymak2012recovery}, while  \cite{friedlander2012recovering} further investigated the stability and robustness of this approach. 
The authors of \cite{mansour2011weighted} extended the analysis of \cite{khajehnejad2009weighted} to multiple sets.  

From a different angle, attempts to analyze the reweighted $\ell_1$ algorithm introduced in \cite{candes2008enhancing} also generated interest in weighted $\ell_1$-minimization in the compressive sensing community. The authors of \cite{xu2010breaking} attempted to give some rigorous guarantees of improvements for the reweighted $\ell_1$-minimization over standard $\ell_1$-minimization, which was lacking in \cite{candes2008enhancing}, by analyzing a two-step reweighted $\ell_1$-minimization scheme.  Further improvement on the analysis and guarantees of the two-step reweighted $\ell_1$-minimization was given in \cite{khajehnejad2010improved}; while a variant of the reweighted $\ell_1$-minimization algorithm was proposed in \cite{mansour2012support}. 

Until recently, theoretical guarantees for weighted $\ell_1$-minimization largely relied on analysis tools of restricted isometry property and null space property that are tailored to the analysis of standard $\ell_1$-minimization without weights.  As such, the recovery guarantees for weighted $\ell_1$-minimization have until recently shown only that weighted $\ell_1$-minimization cannot do worse than standard $\ell_1$-minimization in cases where the weights are aligned with the sparsity structure.  
Recently, the paper \cite{rauhut2015interpolation}, focusing on function interpolation, proved that weighted $\ell_1$-minimization can provide sample complexity $m \geq C k \text{log}(N/k)$ for certain classes of structured random matrices where $\ell_1$-minimization cannot, namely, matrices arising from orthonormal systems which are not uniformly bounded in the $L_{\infty}$ norm.  To achieve these results, they associate to a set of weights $\{ \om_j \}_{j \in[N]}$ with $\om_j \geq 1$   the class of what they called \emph{weighted} $s$-sparse signals: the weighted sparsity of a vector $\x \in \mathbb{R}^N$ is given by $\| \x \|_{\omega,0} = \sum_{\{j\in [N] ~:~ x_j \neq 0\}} \om_j^2$. %
Since $\om_j \geq 1$, this denotes a smaller class of signals than the class of $k$-sparse signals.   Along with this model class, they introduced a certain \emph{weighted} null space property to analyze the performance of weighted $\ell_1$-minimization which, in turn, is strictly weaker than the standard null space property required for $\ell_1$-minimization.
Still, they did not prove that the tools of weighted sparsity and weighted null space property can be used to derive sparse recovery guarantees that beat the $m \geq C k \text{log}(N/k)$ complexity of standard $\ell_1$-minimization.
A particularly interesting choice of weights suggested in \cite{rauhut2015interpolation} and also used in this work is the polynomially growing weights, i.e. $\om_j \leq j^\a$ for $\a \in (0,1/2)$. This model of weights relates to the applications of wavelet sparsity and imaging, where essentially one would use polynomially growing weights in light of the sparsity distribution. Other works related to the above include \cite{peng2014weighted,bouchot2015compressed}.
 
 Later, Mansour and Saab in \cite{mansour2014recovery} used a closely related weighted null space property and analyzed the performance of weighted $\ell_1$-minimization in the \emph{two weights} regime: $\om_j = w_1 \geq 0$ on a subset $\suppest \subset [N]$, and $\om_j = w_2 = 1 \geq w_1$ on the complement, $\suppest^c$. Note that this weighting scheme is equivalent to $\om_j =w_1 \geq 1$ on $\suppest$, and $\om_j = w_2 \geq w_1$ on $\suppest^c$, as rescaling the weighted $\ell_1$-minimization objective \eqref{eqn:wl1min} does not change the set of minimizers, if the scaling is global on the signal components, i.e., if $w_2/w_1$ is constant. For a similar reason, without loss of generality one may assume $w_1 = 1$. Observe that the set $\suppest$ can be interpreted as the set of indices which one has prior reason to believe the signal will be localized, and the relative proportion $w_2/w_1$ on $\suppest^c$ corresponds to a confidence in this prior belief.  Through the weighted null space property, they provided in the noiseless setting, $\eta = 0$, a reduced sampling complexity for a fixed $k$-sparse signal $\x$ supported on a set $\support$ of $m \geq C\left(\rho^2 k + n\right)\log(N/k)$ for $n=|\{ \support^c \cap \suppest \} \cup \{ \support \cap {\suppest^c} \}| \leq k$ and $\rho = w_1/w_2 \in [0,1]$ such that 
as the estimated support $\suppest$ converges towards the true support $\support,$ and as the confidence grows as $\rho \rightarrow 0$, the sample complexity improves. Their results hold for a fixed support $\support$ hence it is referred to as {\em nonuniform}, but they showed that {\em uniform} (holding $\forall \support$) results will require the ``baseline'' $m \geq C k \text{log}(N/k)$ but this bound grows as a function of the error in the support estimate as opposed to the sparsity. For more details, see Section \ref{sec:sample}.

\subsection{Contribution of this work}\label{sec:mresults}
 
The purpose of this work is to provide a unified framework for the sample complexity weighted sparse recovery via weighted $\ell_1$-minimization, using the tools of weighted sparsity and weighted null space property.  What is perhaps surprising is that we find that as long as the weights in the weighted $\ell_1$-minimization program grow mildly, namely, if for some $\gm > 0$ there is an ordering of the weights satisfying $\om_{j} \geq \max\{1, \sqrt{2\gm\log(j/s)} \}$, and as long as the underlying signal is supported on a set $\support \subset [N]$ of \emph{weighted  cardinality} $\om(\support) := \sum_{j \in \support} \om_j^2 \leq s$, then a number of linear Gaussian measurements $m = \bigO({s}/{\gm})$ suffices for robust recovery of the underlying signal. We show that special cases of this main result range from: 
\begin{itemize}
 \item[a)] the standard sparse recovery setting where all weights $\om_j \equiv 1$, the unweighted sparsity $k$ equals the weighted sparsity $s$, and $m = \bigO\left(k \text{log}(N/k)\right)$
 \item[b)] the setting where the support is known a priori and $m = \bigO(k)$ Gaussian measurements suffice
 \item[c)] the setting of sparse recovery with prior information and using two weights, where $m$ depends on how well the weights are aligned with the support set $\support$; hence we re-derived the results from  \cite{mansour2014recovery}, but also extended these results to the setting of additive noise 
\end{itemize}
Similar to \cite{mansour2014recovery} our results are {\em nonuniform} that is for a fixed support (which we do not need to know a priori), and the weights on $\support$ do not all have to be smaller than the weights on $\support^c$ for our recovery results to hold.

Underlying our result is a computation for the Gaussian mean width of classes of weighted sparse signals.
The Gaussian mean width is a notion of the complexity of a set and is used to estimate the sample complexity for reconstruction with Gaussian matrices. Thus we use it to derive sample complexity bounds for weighted $\ell_1$-minimization over a set of $s$-weighted sparse signals, which constitutes a subset of the set of unweighted $s$-sparse signals and thus has smaller Gaussian width. This gives a quantitative upper bound (as stated in Theorem \ref{thm:main}) on the sample complexity which depends on how large the weights are outside of the support set of $\x$, and ranges from $k$ to $k\log(N/k)$ depending on the choice of weights.

\section{Preliminaries} \label{sec:prelim}

\subsection{Notation \& definitions}
Scalars are denoted by lowercase letters (e.g. $k$), vectors by lowercase boldface letter (e.g., ${\bf x}$), sets by uppercase calligraphic letters (e.g., $\mathcal{S}$) and matrices by uppercase boldface letter (e.g. ${\bf A}$).
The cardinality of a set $\mathcal{S}$ is denoted by $|\mathcal{S}|$ and $[N] := \{1, \ldots, N\}$.
Given $\mathcal{S} \subseteq [N]$, its complement is denoted by $\mathcal{S}^c := [N] \setminus \mathcal{S}$ and $\x_\mathcal{S}$ is the restriction of $\x \in \RR^N$ onto $\mathcal{S}$, i.e.~ $(\x_\mathcal{S})_i = \x_i$ if $i \in \mathcal{S}$ and $0$ otherwise. 

The $\ell_p$-norm of a vector ${\bf x} \in \RR^N$ is defined as $\|{\bf x}\|_p := \left ( \sum_{i=1}^N x_i^p \right )^{1/p}$; while
the weighted $\ell_1$-norm is $\|\x\|_{\om,1}:= \sum_{i=1}^N \om_i|x_i|$.

As in the works of \cite{candes2008enhancing,mansour2012support,rauhut2015interpolation}, we consider the setting of weighted $\ell_1$-minimization with a fixed weight $\om_j > 0$ corresponding to each index $j \in [N]$.   Corresponding to such a set of weights, the paper \cite{rauhut2015interpolation} introduced the notion of weighted cardinality of a subset ${\cal S} \subset [N]$ as a generalization of the standard cardinality:
\begin{definition}
\label{def:swsparsity}
Fix $s \in \mathbb{N}$ and a set of weights $\om_j \geq 1$.  A subset ${\cal S} \subset [N]$ is said to have weighted cardinality $s$ if 
\begin{equation}
 \label{eqn:swsparsity}
  \om(\support) := \sum_{j\in\support} \om_j^2 \leq s.
 \end{equation} 
 \end{definition}
Correspondingly, the weighted sparsity of a signal $\x\in \RR^N$, denoted by $\| \x \|_{\om,0} = \sum_{j:\{x_j \neq 0\}} \om_j^2$. Note that for constant weights $\om_j \equiv 1$, the weighted cardinality of a subset coincides with the standard cardinality; while the weighted sparsity of a signal coincides with the standard sparsity.  Otherwise, the set of subsets of $[N]$ of weighted cardinality $s$ constitutes a strict subset of those subsets of cardinality $s$.  In turn, the set of $s$-weighted sparse signals, or those $\x \in \RR^N$ supported on a subset ${\cal S}$ of weighted cardinality $\om(\support) \leq s$, denotes a strict subset of $s$-sparse signals.

\section{Weighted sparse recovery with noisy Gaussian measurements} \label{sec:robustness}

Our main results on the sample complexity of weighted sparse approximation using Gaussian matrices is based on a calculation for the Gaussian mean width of the set of weighted sparse signals.  

In the ensuing analysis we will use the following nonuniform version of the uniform weighted robust null space property ($\w$-RNSP) adapted from Definition 4.1 in \cite{rauhut2015interpolation}.
 \begin{definition}[Nonuniform $\w$-RNSP, \cite{rauhut2015interpolation}]
\label{def:nuwnspgeneral}
Given a weight vector $\w$, a matrix $\A\in\RR^{m\times N}$ is said to satisfy the $\w$-RNSP of order $s$ relative to a fixed set , $\support\subset[N]$, with $\om(\support)\leq s$ and parameters $\r \in (0,1]$ and $\tau > 0$ if 
\begin{equation}
\label{eqn:nuwnspgeneral}
\|\z_{\support}\|_{\om,1} \leq \r\|\z_{\support^c}\|_{\om,1} + \tau\|\A\z\|_2, ~\mbox{ for all } ~\z\in\RR^N.
\end{equation}
\end{definition}
The standard RNSP where the $\ell_{\om,1}$ is replaced by the $\ell_{1}$ in \eqref{eqn:nuwnspgeneral} above is a sufficient condition for robust sparse recovery with $\ell_1$-minimization \eqref{eqn:l1min}, see \cite{foucart2013mathematical} and references therein. Similarly, $\w$-RNSP is also a sufficient condition for robust weighted $\ell_1$-minimization \eqref{eqn:wl1min}, see \cite{rauhut2015interpolation,mansour2014recovery}.

Our derivation relies on the work of Chandrasekaran et al. \cite{chandrasekaran2012convex}, where sufficient conditions for robust recovery are provided based on the tangent cone of a convex function at the signal to be recovered, see also \cite{tropp2014convex}. The tangent cone  (also known as the descent cone) of the $\ell_{\om,1}$-norm is defined as follows. 
\begin{equation*}
\label{eqn:tgntcone}
\mathcal{T}_{\om}(\x) := \mbox{cone}\{\mathbf{u}-\x ~: \mathbf{u}\in\RR^N, ~\|\mathbf{u}\|_{\om,1} \leq \|\x\|_{\om,1} \}\,,
\end{equation*}
where ``cone'' refers to the conic hull of the indicated set. The set $\mathcal{T}_{\om}(\x)$ consists of the directions from $\x$, which do not increase the value of $\|\x\|_{\om,1}$.

However, we use the approach used in \cite{mansour2014recovery} where instead of using the tangent cone defined above, we use the convex cone of the following set, which contains the set of vectors that violate the $\w$-RNSP.
\begin{equation}
\label{eqn:uwsnsppgaussfailsetws}
\ZZ_{\om,\support} := \left\{\z\in\RR^N ~: ~ \|\z_{\support}\|_{\om,1} \geq \|\z_{\support^c}\|_{\om,1}\right\}\,,
\end{equation}
for $\support\subset [N] ~\mbox{with} ~\om(\support) \leq s$.

Using this set, we now provide a simple lemma concerning the robust recovery of weighted sparse signals. 
\begin{lemma}
\label{lem:nurws}
Consider a vector of weights $\w \geq 1$, and a signal $\x\in\RR^N$ with an unknown support set $\support\subset [N]$ of weighted cardinality $\om(\support) \leq s$.  Let  $\A\in\RR^{m\times N}$ be a measurement matrix, and assume that noisy measurements are taken,  $\obs = \A\x + \noise$ with $\|\noise\|_2\leq \eta$. Fix $\zeta = \zeta(\om,s) >0$ such that
\begin{equation}
\label{eqn:nugetmrws}
\inf_{\z\in\ZZ_{\om,\support},\|\z\|_2 = 1} \|\A\z\|_2 \geq \zeta\,.
\end{equation}
Then any solution $\sol$ of \eqref{eqn:wl1min} satisfies
\begin{align}
\label{eqn:errorbound}
\|\sol - \x\|_2 \leq 2\eta/\zeta\,, \quad \mbox{ and } \quad \|\sol - \x\|_{\omega,1} \leq 4 \eta \sqrt{s} /\zeta\,.
\end{align}
\end{lemma}

\begin{proof}
By the definition of $\sol$ in \eqref{eqn:wl1min}, we have $\|\sol\|_{\om,1} \leq \|\x\|_{\om,1}$.  Recalling that $\x = \x_{\support}$, this implies
\begin{align}
\| \sol_{S^c} \|_{\om,1} \leq \| \x \|_{\om,1} - \| \sol_{S} \|_{\om,1} = \| \x_{\support} \|_{\om,1} - \| \sol_{S} \|_{\om,1} \leq  \| (\x - \sol)_{\support} \|_{\om,1},
\end{align}
that is, that $(\x - \sol) \in\ZZ_{\om,\support}$.
We now apply the triangle inequality to get
\begin{equation}
\label{eqn:prfthmnurws}
\|\A(\x-\sol)\|_2 \leq \|\A\sol-\obs\|_2 + \|\A\x-\obs\|_2 \leq 2\eta\,.
\end{equation}
By \eqref{eqn:nugetmrws} we have $\|\A(\x-\sol)\|_2 \geq \zeta\|\x-\sol\|_2$. This combined with \eqref{eqn:prfthmnurws} gives the $\ell_2$-norm error bound in \eqref{eqn:errorbound}.   To obtain the weighted $\ell_1$-norm error bound, note by Cauchy-Schwarz that
\begin{equation}
\label{eqn:wl1bound}
\| (\x - \sol)_{\support} \|_{\om,1} \leq \sqrt{\sum_{j \in \support} \om_j^2} \| (\x - \sol)_{\support}  \|_2 \leq 2 \eta \sqrt{s}/ \zeta\,.
\end{equation}
As $(\x - \sol) \in\ZZ_{\om,\support}$, it follows that 
\begin{align}
\label{eqn:wl1bound2}
\| (\x - \sol) \|_{\om,1} \leq \| (\x - \sol)_{\support^c} \|_{\om,1} + \| (\x - \sol)_{\support} \|_{\om,1} \leq 2 \| (\x - \sol)_{\support} \|_{\om,1} \leq  4\eta \sqrt{s}/ \zeta\,. \nonumber
\end{align}
Thus concluding the proof.
\end{proof}

\subsection{Gordon's escape through mesh theorem}\label{sec:gordon}

According to \eqref{eqn:nugetmrws}, successful recovery of a weighted sparse signal is achieved, when the minimal gain of the measurement
matrix over intersection of the sphere and the set $\ZZ_{\om,\support}$ is greater than some positive constant. For Gaussian matrices the probability of this event can be estimated by Gordon's escape through mesh theorem \cite{gordon1988milman}, see also [Theorem 9.21,\cite{foucart2013mathematical}]. We begin by
defining the Gaussian width of a set, which plays a key role in the proof of the theorem and our analysis.
\begin{definition}
\label{def:gwidth}
Given a standard Gaussian vector, $\g\in\RR^m$ with $g_i\sim \NN(0,1)$, the Gaussian width of a set $\X \subset\RR^m$, denoted as $\ell(\X)$, is defined as
 \begin{equation}
\label{eqn:gwidth}
\ell(\X) := \E\sup_{\x\in\X} \langle \g,\x \rangle.
\end{equation}
\end{definition}
 
 \pagebreak
\begin{theorem}[Gordon's escape through mesh, \cite{gordon1988milman}]
\label{thm:gordon}
Consider an $m\times N$ matrix $\A$ with $a_{ij}\bfrac{\bfrac{\mbox{\scriptsize i.i.d}}{\mathlarger{\mathlarger{\sim}}}}{}\NN(0,1)$, a standard Gaussian random vector $\g\in\RR^m$ and a set $\X\subset S^{N-1}$. For $t > 0$ we have
\begin{equation}
\label{eqn:gordon}
\P \left( \inf_{\x\in\X} \|\A\x\|_2 \leq \E\|\g\|_2 - \ell(\X) - t \right) \leq e^{-t^2/2}.
\end{equation}
\end{theorem}
\noindent We will also require the following standard result about the mean length of a standard Gaussian vector, $\g\in\RR^m$.
 \begin{equation}
\label{eqn:meanlength}
\E\|\g\|_2 = \frac{\sqrt{2}\Gm\left((m+1)/2\right)}{\Gm\left(m/2\right)}, \quad \mbox{and} \quad \frac{m}{\sqrt{m+1}} \leq \E\|\g\|_2 \leq \sqrt{m}\,,
\end{equation}
where $\Gm$ is the gamma function.

\subsection{Gaussian mean width for the set of weighted sparse signals}\label{sec:gwidth}

To apply Theorem \ref{thm:gordon}, we need to estimate $\ell\left(\ZZ_{\om,\support} \cap S^{N-1}\right)$, which is given in the following lemma. 
\begin{lemma}
\label{lem:gausswidth}
Fix weights $\om_j \geq 1$ and consider a subset $\support \subset [N]$ with $\om(\support) \leq s$ and  $| \support | \leq k \leq s$.  Let the set $\ZZ_{\om,\support}$ be as in \eqref{eqn:uwsnsppgaussfailsetws}. Then
\begin{equation}
\label{eqn:gausswidth}
\ell\left(\ZZ_{\om,\support} \cap S^{N-1}\right) \leq  \sqrt{k} +  \inf_{h \geq 0} \left( h\sqrt{s}  + \left( \sqrt{\frac{2}{\pi e}} \sum_{j\in\support^c} \frac{e^{-h^2 \om_j^2/2}}{h^2 \om_j^{2}}\right)^{1/2} \right)\,.
\end{equation}
\end{lemma}

\begin{proof}
We start by replacing the $\sup$ by the $\max$ in the definition of $\ell\left(\ZZ_{\om,\support} \cap S^{N-1}\right)$ in \eqref{eqn:gwidth}, since $\ZZ_{\om,\support} \cap S^{N-1}$ is compact. Observe that 
\begin{equation*}
\label{eqn:cvxconelogic}
\max_{\z\in\ZZ_{\om,\support} \cap S^{N-1}} \langle \z,\g\rangle = \max_{\z\in\ZZ_{\om,\support} \cap S^{N-1}\cap\{\z ~: ~z_j \geq 0\}} \sum_{j=1}^N |g_j|z_j\,.
\end{equation*}
Define the {\em convex cone}
\begin{equation}
\label{eqn:cvxcone}
\widetilde{\ZZ}_{\om,\support} := \ZZ_{\om,\support} \cap\{\z ~: ~z_j \geq 0\} 
\end{equation} 
and its dual, denoted as $\widetilde{\ZZ}^*_{\om,\support}$, is defined: 
\begin{equation}
\label{eqn:cvxconedual}
\widetilde{\ZZ}^*_{\om,\support} := \left\{\u\in\RR^N ~: ~\langle \u,\z \rangle \geq 0 ~\forall \z\in\ZZ_{\om,\support}\right\}.
\end{equation}
We apply {\em duality} to bound the Gaussian width, $\ell\left(\ZZ_{\om,\support} \cap S^{N-1}\right)$, as follows
\begin{align}
\label{eqn:dualitybound0}
\ell\left(\ZZ_{\om,\support} \cap S^{N-1}\right) & = \E \max_{\z\in\widetilde{\ZZ}_{\om,\support}\cap S^{N-1}} \langle \z,\widetilde{\g}\rangle \\
\label{eqn:dualitybound}
& \leq \E \min_{\u\in\widetilde{\ZZ}^*_{\om,\support}} \|\u + \widetilde{\g}\|_2\,,
\end{align}
where $\widetilde{\g}$ has components $|g_j|, ~j\in [N]$. 
A linear program in maximization linear form 
\begin{equation}
\label{eqn:primal}
\max ~\langle{\bf c}, \x\rangle \quad \mbox{s.t.} \quad \A\x \leq {\bf b}, ~\x\geq 0
\end{equation}
like in \eqref{eqn:dualitybound0}, referred to as the primal linear program, 
has its dual linear program
\begin{equation}
\label{eqn:pdual}
\min ~\langle{\bf b}, \obs\rangle \quad \mbox{s.t.} \quad \A^T\obs \geq {\bf c}, ~\obs\geq 0
\end{equation}
as in \eqref{eqn:dualitybound}. Bound  \eqref{eqn:dualitybound} is due to the weak duality principle, which states that the optimum of the dual is an upper bound to the optimum of the primal. For details about duality theory, see Appendix B.5 of \cite{foucart2013mathematical}.

To use the above bound we need to characterize the dual cone, $\widetilde{\ZZ}^*_{\om,\support}$, as follows for instance.

\begin{proposition}
\label{pro:setconstruct}
The dual cone of $\widetilde{\ZZ}_{\om,\support}$ satisfies  
\begin{equation}
\label{eqn:setconstruct}
\widetilde{\ZZ}^*_{\om,\support} \supset \VV_{\om,\support} := \bigcup_{h\geq 0}\left\{ \u \in \RR^N : u_j = h\om_j \mbox{ for } j\in\support, ~u_j \geq -h\om_j ~\mbox{for} ~j\in\support^c \right\}.
\end{equation}
\end{proposition}

\begin{proof}
For any $\z\in\widetilde{\ZZ}_{\om,\support}$, any $\u\in \VV_{\om,\support}$, and any $h\geq 0$ we have 
\begin{align}
\label{eqn:setconstructprf1}
\langle \u,\z \rangle = \sum_{j\in\support} u_j z_j + \sum_{j\in\support^c} u_j z_j & \geq h\sum_{j\in\support} \om_j z_j - h\sum_{j\in\support^c} \om_j z_j \nonumber\\ 
& = h\left(\|\z_{\support}\|_{\om,1} - \|\z_{\support^c}\|_{\om,1}\right) \geq 0.
\end{align}
The first equality follows from the definition of $\langle \cdot,\cdot \rangle$, while the first inequality follows from the definition of $\VV_{\om,\support}$. The second equality is due to the definition of $\|\cdot\|_{\om,1}$, while the last inequality is due to the definition of $\widetilde{\ZZ}_{\om,\support}$. This shows that $\u\in \widetilde{\ZZ}^*_{\om,\support}$.
\end{proof}

Therefore, using \eqref{eqn:dualitybound} with Proposition \ref{pro:setconstruct} we bound $\ell\left(\ZZ_{\om,\support} \cap S^{N-1}\right)$ as follows.
\begin{align}
\label{eqn:gwidthbound1}
\ell\left(\ZZ_{\om,\support} \cap S^{N-1}\right) 	& \leq \E\min_{\u\in\VV_{\om,\support}} \|\u + \widetilde{\g}\|_2 \\
\label{eqn:gwidthbound2}
									& \leq \E\min_{\u\in\VV_{\om,\support}} \|\u_{\support^c}  + \widetilde{\g}_{\support^c} \|_2 + \E\|{\bf v}_{\support} + \widetilde{\g}_{\support}\|_2 \\
\label{eqn:gwidthbound3}
									& \leq \E\min_{u_j\geq -h\om_j} \left(\sum_{j\in\support^c} |\widetilde{g}_j + u_j|^2\right)^{1/2} + \E\|{\bf v}_{\support}\|_2  + \E\|\widetilde{\g}_{\support}\|_2 \\
\label{eqn:gwidthbound4}
									& \leq \E\min_{u_j\geq -h\om_j} \left(\sum_{j\in\support^c} |\widetilde{g}_j + u_j|^2\right)^{1/2} + h\sqrt{\om(\support)}  + \sqrt{k}\,.
\end{align}
Inequality \eqref{eqn:gwidthbound2} and  \eqref{eqn:gwidthbound3} are due to the application of the triangle inequality and are true for any ${\bf v}\in\VV_{\om,\support}$ and $h\geq 0$. The first term in \eqref{eqn:gwidthbound4} is due to the fact that $u_j = h\om_j$ for $j\in\support$ and the definition of the weighted sparsity, the second term comes from the upper bound in \eqref{eqn:meanlength}. 
Next we upper bound the first term of \eqref{eqn:gwidthbound4} as follows.
\begin{equation}
\label{eqn:gwidthbound5}
 \E\min_{u_j\geq -h\om_j} \left(\sum_{j\in\support^c} |\widetilde{g}_j + u_j|^2\right)^{1/2} 
 \leq  \E \left(\sum_{j\in\support^c} S_{h\om_j}(g_j)^2\right)^{1/2} 
\leq \left(\E\sum_{j\in\support^c} S_{h \om_j}(g_j)^2\right)^{1/2}\,,
\end{equation}
where $g_j$ is a component of the standard Gaussian vector $\g$, and $S_{h\om_j}(\cdot)$ is the soft-thresholding operator,
\begin{equation}
\label{eqn:softthres}
S_\lambda(x) := \mbox{sign}(x) \left[ \max(0,|x| - \lambda) \right]\,.
\end{equation}
The last inequality in \eqref{eqn:gwidthbound5} comes from the application of H\"older's inequality.
We use the symmetry of Gaussian random variable $g$ and the soft-thresholding operation with a change of variable to compute $\E S_{\lambda}(g)^2$, as derived in Section 9.2 of \cite{foucart2013mathematical}:
\begin{align}
\label{eqn:gwidthbound6a}
\E S_{\lambda}(g)^2  & = \frac{1}{\sqrt{2\pi}} \int_{-\infty}^{\infty} S_{\lambda}(u)^2 e^{-u^2/2}du = \frac{2}{\sqrt{2\pi}} \int_{0}^{\infty} S_{\lambda}(u)^2 e^{-u^2/2}du \\
\label{eqn:gwidthbound6b}
& = \sqrt{\frac{2}{\pi}} \int_{\lambda}^{\infty} (u-\lambda)^2 e^{-u^2/2}du = e^{-\lambda^2/2}\sqrt{\frac{2}{\pi}} \int_{0}^{\infty} v^2 e^{-v^2/2}e^{-\lambda v}dv \,.
\end{align}
Using the bound $v e^{-v^2/2} \leq e^{-1/2},$ we can upper bound $\E S_{\lambda}(g)^2$ further from \eqref{eqn:gwidthbound6b},
\begin{align}
\label{eqn:gwidthbound6c}
\E S_{\lambda}(g)^2 & \leq e^{-1/2}e^{-\lambda^2/2}\sqrt{\frac{2}{\pi}} \int_{0}^{\infty} ve^{-\lambda v}dv = \sqrt{\frac{2}{\pi e}} \lambda^{-2}e^{-\lambda^2/2} \,.
\end{align}

Finally, we upper bound the right hand side (RHS) expression of \eqref{eqn:gwidthbound5} as 
\begin{equation}
\label{eqn:gwidthbound6d}
\left(\sum_{j\in\support^c} \E S_{h \om_j}(g_j)^2\right)^{1/2} \leq \left( \sqrt{\frac{2}{\pi e}} \sum_{j\in\support^c} \frac{e^{-h^2 \om_j^2/2}}{h^2 \om_j^{2}}\right)^{1/2}\,. \nonumber
\end{equation}
We have arrived at an $h$-dependent bound for the Gaussian mean width \eqref{eqn:gwidthbound4}.  Taking the infimum of the bound over $h \geq 0$ provides the lemma.
\end{proof}

We now apply the bound in Lemma \ref{lem:gausswidth} to derive a general result on the sample complexity of weighted sparse recovery using Gaussian random measurements. 
\begin{lemma}
\label{lem:nuscgaussrws}
Fix weights $\om_j \geq 1$, and consider a signal $\x\in\RR^N$ supported on $\support\subset [N]$ with $|\support|\leq k$ and $\om(\support) = \sum_{j \in \support} \om_j^2 \leq s$.  Let $\A\in\RR^{m\times N}$ with $a_{ij}\bfrac{\bfrac{\mbox{\scriptsize i.i.d}}{\mathlarger{\mathlarger{\sim}}}}{}\NN(0,1)$, and fix $\zeta >0$.  
Assume that noisy measurements are taken,  $\obs = \A\x + \noise$ with $\|\noise\|_2\leq \eta$. If, for some $0<\d<1$,
\begin{equation}
\label{eqn:nuscgaussrws}
\frac{m}{\sqrt{m+1}} \geq  \sqrt{k} + \sqrt{2\log\left(\frac{1}{\d}\right)} + \zeta + 
\inf_{h > 0} \left( h\sqrt{s}  + \left( \sqrt{\frac{2}{\pi e}} \sum_{j\in\support^c} \frac{e^{-h^2 \om_j^2/2}}{h^2 \om_j^{2}}\right)^{1/2} \right),
\end{equation}
Then with probability exceeding $1-\d$, any solution $\sol$ of \eqref{eqn:wl1min} satisfies
\begin{align}
\label{eqn:errorbound2}
\|\sol - \x\|_2 \leq 2\eta / \zeta, \quad \mbox{and} \quad \|\sol - \x\|_{\om,1} \leq 4\sqrt{s}\eta/\zeta\,.
\end{align}
\end{lemma}

The proof of Lemma \ref{lem:nuscgaussrws} in light of Lemma \ref{lem:gausswidth} is by now standard, see for example Theorem 3 of \cite{kabanava2014robust}, or more generally in \cite{chandrasekaran2012convex}. We therefore skip the details of the proof and provide only a sketch. 

\begin{proof}[Sketch of proof of Lemma \ref{lem:nuscgaussrws}]
For $t>0$, we want to determine
\begin{equation}
\label{eqn:gordon2b}
\P \left( \inf_{\z\in\ZZ_{\om,\support}\cap S^{N-1}} \|\A\z\|_2 \leq \E\|\g\|_2 - \ell\left(\ZZ_{\om,\support}\cap S^{N-1}\right) - \zeta - t \right) \leq e^{-t^2/2}.
\end{equation}

We have a lower bound of $\E\|\g\|_2$ and an upper bound of $\ell\left(\ZZ_{\om,\support} \cap S^{N-1}\right)$ in \eqref{eqn:meanlength} and \eqref{eqn:gausswidth} respectively, thus the task remains to choose an appropriate value for $t$ so that we have
\begin{equation}
\label{eqn:gordon3b}
\E\|\g\|_2 - \ell\left(\ZZ_{\om,\support}\cap S^{N-1}\right) - \zeta - t \geq 0. \nonumber
\end{equation}
Setting $t=\sqrt{2\log(1/\d)}$ upper bounds the probability \eqref{eqn:gordon2b} by $\d$, leading to the required sampling bound \eqref{eqn:nuscgaussrws} with probability at least $1-\d$ since 
\begin{align}
\label{eqn:gordon2c}
\P \left( \inf_{\z\in\ZZ_{\om,\support}\cap S^{N-1}} \|\A\z\|_2 \geq \zeta \right) & \geq \P \left( \inf_{\z\in\ZZ_{\om,\support}\cap S^{N-1}} \|\A\z\|_2 \geq \E\|\g\|_2 - \ell\left(\ZZ_{\om,\support}\cap S^{N-1}\right) - t \right) \nonumber \\ 
& \geq 1-\d\,, \nonumber
\end{align}
This concludes the proof.
\end{proof}

\section{Sample complexity for weighted sparse approximation}\label{sec:sample}

With Lemmas \ref{lem:gausswidth} and \ref{lem:nuscgaussrws} in hand, we are now ready to state the main result: under very mild growth conditions on the weights $\om_j$, the sample complexity of weighted sparse recovery given in Lemma \ref{lem:nuscgaussrws} is linear with respect to the sparsity of the signal.

\begin{theorem}
\label{thm:main}
Fix weights $\om_j \geq 1$, and suppose that $\gm > 0$ is such that the weights $\mathop{\om_{j} \geq \max\left\{1, \sqrt{2\gm\log({j}/{s})}\right\}, ~\forall j \in [N]}$.  Fix $\A\in\RR^{m\times N}$ with $a_{ij}\bfrac{\bfrac{\mbox{\scriptsize i.i.d}}{\mathlarger{\mathlarger{\sim}}}}{}\NN(0,1)$ and $\mathop{\zeta > 0}$.  
Consider a signal $\x\in\RR^N$ supported on $\support \subset [N]$ with $|\support|\leq k$ and  $\mathop{\om(\support) = \sum_{j \in \support} \om_j^2 \leq s}$. Assume that noisy measurements are taken,  $\obs = \A\x + \noise$ with $\|\noise\|_2\leq \eta$. Suppose, for some $0<\d<1,$
 \begin{equation}
\label{eqn:nuscgaussrws1}
\frac{m}{\sqrt{m+1}} \geq \sqrt{ \frac{2s}{\gm}}  + \sqrt{k} + \sqrt{\gm s} + \sqrt{2\log\left(\frac{1}{\d}\right)} + \zeta\,.
\end{equation}
Then with probability exceeding $1-\d$, any solution $\sol$ of \eqref{eqn:wl1min} satisfies
\begin{align}
\label{eqn:errorbound21}
\|\sol - \x\|_2 \leq  2 \eta /\zeta\,, \quad \mbox{and} \quad \|\sol - \x\|_{\om,1} \leq 4\sqrt{s}\eta/\zeta\,.
\end{align}
\end{theorem}
\noindent This theorem provides a sharp bound on the sample complexity for weighted sparse approximation in terms of how well the weights used in the weighted $\ell_1$ reconstruction program align with the signal support.   The more that  the support of $\support$ coincides with a subset of indices having small relative weight (corresponding to smaller weighted sparsity $s$),  and the more that the complementary support $\support^c$ coincides with indices having larger weight (corresponding to larger value of $\gm$), the smaller the necessary number of measurements $m$ to recover a signal supported on $\support$.  Note that the opposite statement is also true: if $\support$ is supported on a set where the weights are large, then the weighted sparsity $s$ becomes large and the sample complexity of weighted sparse approximation becomes \emph{higher} than that for $\ell_1$-minimization with uniform weights.


\noindent We can simplify the bound on the number of measurements appearing in \eqref{eqn:nuscgaussrws1}: using the property that for any positive numbers $A, B, C$ and $D$, 
\begin{equation}
\label{eqn:sqrtbound}
\sqrt{D} \geq \sqrt{A} + \sqrt{B} + \sqrt{C} ~\Rightarrow ~D \geq 2A + 2B + 2C\,.
\end{equation}
and ignoring the terms depending on the robustness and confidence parameters $\zeta$ and $\d$, we have the  sampling rate of
\begin{equation}
\label{eqn:unumsamplesrws}
m = \bigO \left({s}/{\gm} \right).
\end{equation}
Indeed, Theorem \ref{thm:main} recovers several known results on the sample complexity of sparse approximation in special cases (up to a universal constant), as we now outline.

\begin{itemize}
\item {\bf Uniform weights.}  In standard $\ell_1$-minimization, the weights are all $\om_j \equiv 1$.  This is a special case of Theorem \ref{thm:main} with $s = k$ and 
$$
\gm = \frac{1}{2 \log(N/k)},
$$
thus recovering the known sample complexity $m = \bigO \left(k \log(N/k)\right)$ for unweighted sparse recovery, \cite{ayaz2014nonuniform,candes2006quantitative,chandrasekaran2012convex}. 

\item {\bf Polynomially growing weights.}  In \cite{rauhut2015interpolation}, the authors considered the case of polynomially growing weights in an application to smooth function interpolation.  While they did not focus on reduced sample complexity for weights as such, they did show that orthonormal function systems which are not uniformly bounded -- and thus do not satisfy the requirements for unweighted sparse recovery at sample complexity $m = \bigO(k \log(N/k) )$ -- nevertheless can recover weighted sparse signals at this sample complexity with weighted $\ell_1$-minimization.  It was however mentioned in a remark in \cite{rauhut2015interpolation} that for weights satisfying polynomial growth, $\om_j \geq j^{\alpha},$ the set of $s$-weighted sparse signals can be reconstructed via weighted $\ell_1$-minimization using $m = \bigO(\alpha^{-1}s\log(s))$ Gaussian measurements.  Since $j^{\alpha} \geq \log(j^{\alpha}) \geq \alpha \log(j/s)$, Theorem \ref{thm:main} implies that in the setting of polynomially growing weights $\om_j \geq j^{\alpha},$ the sample complexity required to recover a weighted $s$-sparse signal scales no worse than $m = \bigO({s}/{\alpha})$.  

\pagebreak
\item {\bf Sparse recovery with prior support estimates.}  In this setting, one a priori splits the support set $[N]$ into two subsets $\suppest$ and $\suppest^c$: $\suppest$, which one has reason to believe includes the support of the underlying signals of interest, one places a smaller weight $w_1 \geq 1$.  On $\suppest^c$, one places a larger weight of $w_2 > w_1$.  The sample complexity thus depends on how well the prior support estimate $\suppest$ agrees with the true underlying support $\support$ and on the confidence $w_2/w_1$ in this belief.  In terms of weighted sparse recovery,  the more disjoint are the true and estimated support sets $\support$ and $\suppest$, the larger will be the weighted sparsity of the signal supported on $\support$.

Normalize without loss of generality so that $w_1 = 1$ and $w_2 := 1/\rho > 1$.   Suppose that $\x$ is supported on a set $\support$ of size $| \support | = k$.  Then $\om_j = 1$ for $j \in \support \cap \suppest$, and $\om_j = 1/\rho > 1$ for $j \in \support \cap \suppest^c$.   It follows that the weighted cardinality of the support is
\begin{equation}
\label{s:prior}
s \equiv \om(\support) = | \support \cap \suppest | + \rho^{-2} | \support \cap \suppest^c |
\end{equation}
At the same time, $\om_j = 1$ for $j \in \support^c \cap \suppest$ and $\om_j = 1/\rho > 1$ for $j \in \support^c \cap \suppest^c.$  Assuming without loss that the indices in $\support^c$ are ordered according to the increasing arrangement of their weights, we find that for $j \in \support^c$,
\begin{equation}
\label{eqn:kappa}
\om_j \geq  \max\left\{ 1, \sqrt{2\gm \log(j/s)} \right\}\,,  \quad \text{with} \quad \gm = \min\left\{1,\frac{1}{2\rho^2 \log(N/s)}\right\}\,.
\end{equation}
Applying Theorem \ref{thm:main} gives thus the sample complexity for sparse recovery with prior support estimates:
\begin{align}
\label{prior:support}
m = \bigO \left( \frac{s}{\gm} \right) & = \bigO \left( s \rho^2 \log(N/s) + s \right) \nonumber \\
&=  \bigO \left(  \left( \rho^2| \support \cap \suppest | +  | \support \cap \suppest^c | \right)\log(N/s) + s \right)\,.
\end{align}
As in \cite{mansour2014recovery}, fix parameters $\alpha, \beta \in [0,1]$, and suppose $| \suppest | = \beta | \support| = \beta k,$ and suppose that $| \support \cap \suppest | = \alpha | \suppest |$.  Then if $2\rho^2\log(N/s) \geq 1$ ($\Rightarrow \rho \geq \frac{1}{\sqrt{2\log(N/s)}} > 0$ if $N\neq s$), the dominant term in the sampling bound in \eqref{prior:support} is bounded above by
\begin{equation}
\label{nusampleboundUS}
m = \bigO \left( \left(\rho^2 k + n\right) \log(N/k) \right),
\end{equation}
where $n = | \support^c \cap \suppest | + | \support \cap \suppest^c |$ represents the mismatch between the true and estimated supports.  This aligns with the sample complexity for sparse approximation with prior belief from \cite{mansour2014recovery}. If $\rho \rightarrow 0$, which necessitates $\rho \leq \frac{1}{\sqrt{2\log(N/s)}}$ implying $2\rho^2\log(N/s) \leq 1$, and as $n \rightarrow 0$, the sample complexity goes to $s$.  That is as the weights $w_2$ are increased relative to the weights $w_1$, and as the mismatch parameter $n\rightarrow 0$, then $m \rightarrow \bigO(s)$.

Moreover, since our results are for the noisy setting, in addition to recovering the results of \cite{mansour2014recovery}, we generalize those results which were shown only in the noiseless setting. 

\item {\bf Known support.}  In case $\x$ is supported on $\support = \suppest$ of size $| \support | = k = s$ estimated exactly, then if $\rho^2$ is sufficiently small such that $\rho^2 \log(N/k) \leq 1$,  the sample complexity reduces to 
$$
m = \bigO(k),
$$
recovering the known sample complexity of solving for a $k$-sparse signal with known support from Gaussian measurements.
\end{itemize}

\begin{proof}[Proof of Theorem \ref{thm:main}]
Suppose without loss of generality that $\support^c = [N]$ and that the weights are chosen to be $\om_j = \max\{1,\sqrt{2\gm\log(j/s)}\}$ for each $j \in [N]$.   Then 
\begin{align}
\frac{e^{-h^2 \om_j^2/2}}{h^2 \om_j^2} \leq h^{-2} e^{-h^2 \om_j^2/2}  \leq  \left\{ \begin{array}{cc} h^{-2} e^{-h^2/2},  & ~j \leq s \cdot e^{1/(2\gm)}, \\
h^{-2} (\frac{j}{s})^{-\gm h^2}, & ~j > s\cdot e^{1/(2\gm)}
\end{array} \right.
\end{align}
Set $h = \sqrt{2/\gm}$.  It follows that
\begin{align}
\sum_{j \in \support^c}\frac{e^{-h^2 \om_j^2/2}}{h^2 \om_j^2} \leq \sum_{j = 1}^N \frac{e^{-h^2 \om_j^2/2}}{h^2 \om_j^2} & \leq \frac{\gm}{2} \left( s e^{-1/(2\gm)} + \sum_{j=e^{1/(2\gm)}}^N \left(\frac{j}{s}\right)^{-2} \right) \nonumber\\
& \leq \frac{3\gm s}{2} e^{-1/(2\gm)}  \leq \frac{3\gm s}{2e}  \,.
\end{align}
The result is then arrived at by substituting this value of $h$ and applying Theorem \ref{lem:nuscgaussrws}. This completes the proof.
\end{proof}

\section{Numerical simulations} \label{sec:numerics}

\subsection{Weighted sparse recovery}
To perform weighted sparse recovery we generate signals from a distribution we refer to as the {\em weighted sparse signal model}.  Signals generated from this distribution are weighted $s$-sparse in expectation. To be precise, the probability for an index to be in the support of the signal is proportional to the reciprocal of the square of the weights assigned to that index. In particular, let's consider weighted $s$-sparse signals with support $\support$, i.e. $\om(\support)\leq s$, we draw indices of the signal with probability of being nonzero as 
\begin{equation}
\label{eqn:signalprob}
\hbox{Prob}\left(x_j \neq 0\right) = {s}/{N}\cdot {1}/{\om_j^2}\,,
\end{equation} 
where $s$ and $N$ are the weighted sparsity and the ambient dimension of the signal, $\x\in\RR^N$ respectively. The weights $\om_j$ will then be used in the weighted $\ell_1$-minimization \eqref{eqn:wl1min} recovery of $\x$. It is not hard to show that the weighted sparsity in expectation is $s$ as shown below.
\begin{equation}
\label{eqn:wsexpectation}
\E\left[\om(\support)\right] = \sum_{j=1}^N \hbox{Prob}\left(x_j \neq 0\right) \cdot \mathbb{I}_{\{x_j \neq 0\}} \cdot \om_j^2 =  \sum_{j=1}^N {s}/{N} = s \,.
\end{equation} 

In the experiments presented below, we draw signals of dimension $N=500$ from the above model, where the nonzero values are randomly generated as scaled Gaussian random variables without any normalization, and draw $m\times N$ sensing matrices with entries i.i.d. from a normal Gaussian. Noisy measurements are drawn with Gaussian noise $\|\noise\|_2 =  10^{-6}$. Then we run simulations with 50 realizations for each problem instance $(s,m,N)$. 

Here we set out to compare sample complexities of weighted sparse recovery using weighted $\ell_1$-minimization and standard sparse recovery using unweighted $\ell_1$-minimization for signals generated from the weighted sparse signal model.
We vary the number of measurements $m$ such that $~m/N\in[0.05,0.5]$; and fix weighted sparsity of the supp($\x$), $\support$, such that $\om(\support)/m = s/m \in \left[\frac{1}{0.05N},2.5\right]$. Then we record $k$ as the largest $|\support|$ for a given $s$. We considered two regimes of weights: $1)$ polynomially growing weights, i.e. $\om_j = j^\theta$ for $\theta>0$ and $j=1,\ldots,N$ (or more generally $\om_j = \pi(j)^\theta$ for a permutation $\pi$); and $2)$ uniformly random generated weights. In both regimes we consider a reconstruction successful if the recovery error in the $\ell_2$-norm is below $10^{-5}$ and a failure otherwise. Then we compute the {\em empirical probabilities} as the ratio of the number of successful reconstructions to the number of realizations. We present below sample complexity comparisons in these two regimes of weights using the phase transition framework in the phase space of $(s/m,m/N)$. Note that in all the figures we standardized the values of $s/m$ in such a way that the standardized $s/m$ is between $0$ and $1$ for fair comparison.

\vspace{-0.1cm}
\begin{itemize}
\item {\bf Polynomially growing weights.}
Precisely, we assign weights that grow polynomially proportional to $\om_j = j^{1/5}$. 
Fig. \ref{fig:ptsurfP} shows surface plots of empirical probabilities for signals drawn from the distribution of weighted sparse signals using polynomially growing weights. The color bars depict the probabilities of recovering ``close enough'' approximation $\sol$ to the true $\x$. By having larger area of yellow (i.e. with probability one), weighted sparse recovery using \eqref{eqn:wl1min}, in the left panel, outperforms generic sparse recovery using \eqref{eqn:l1min}, in the right panel. 
\vspace{-2mm}\begin{figure}[h!]
\centering
\includegraphics[width=0.43\textwidth,height=0.31\textwidth]{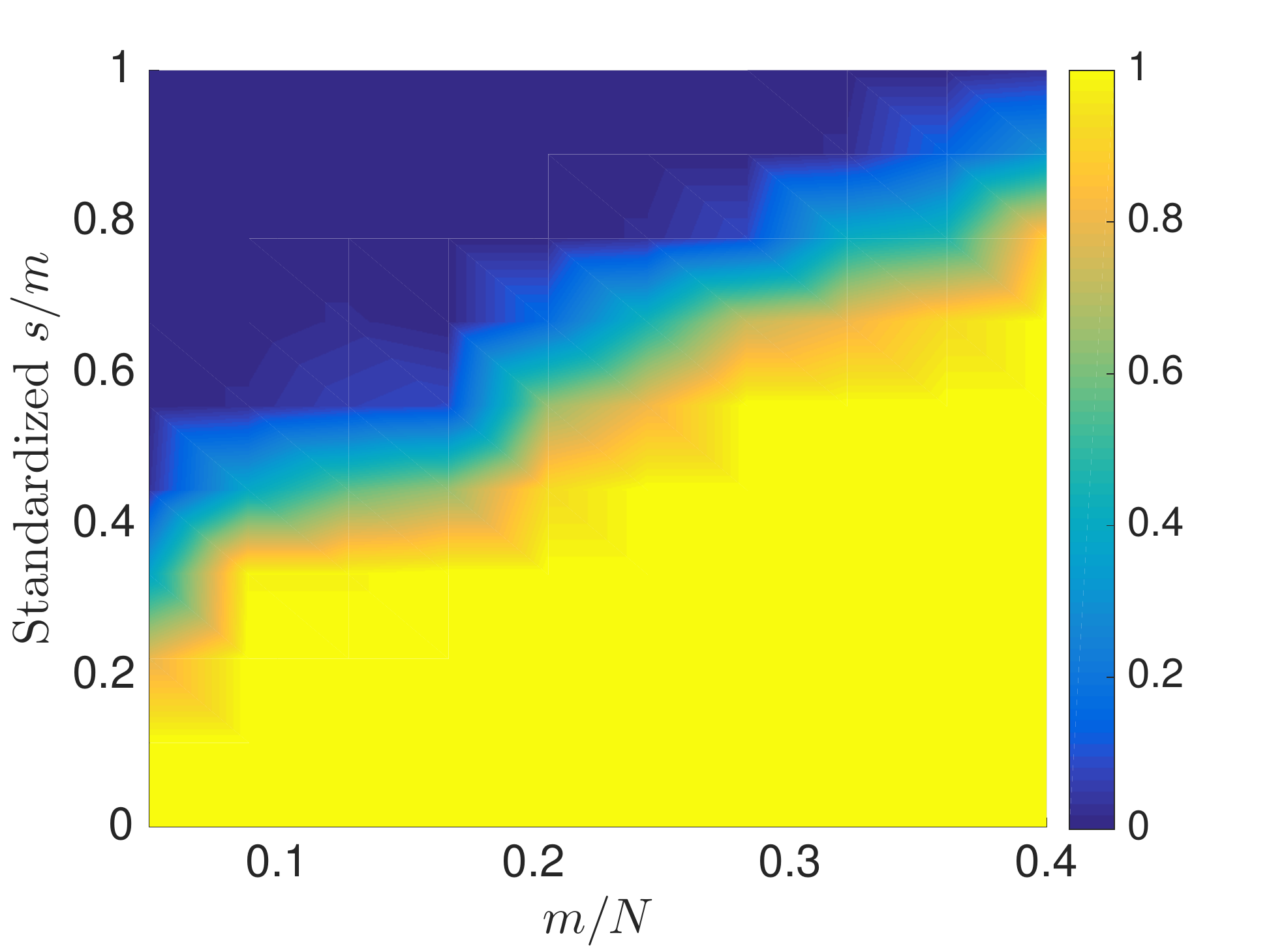} 
\includegraphics[width=0.43\textwidth,height=0.31\textwidth]{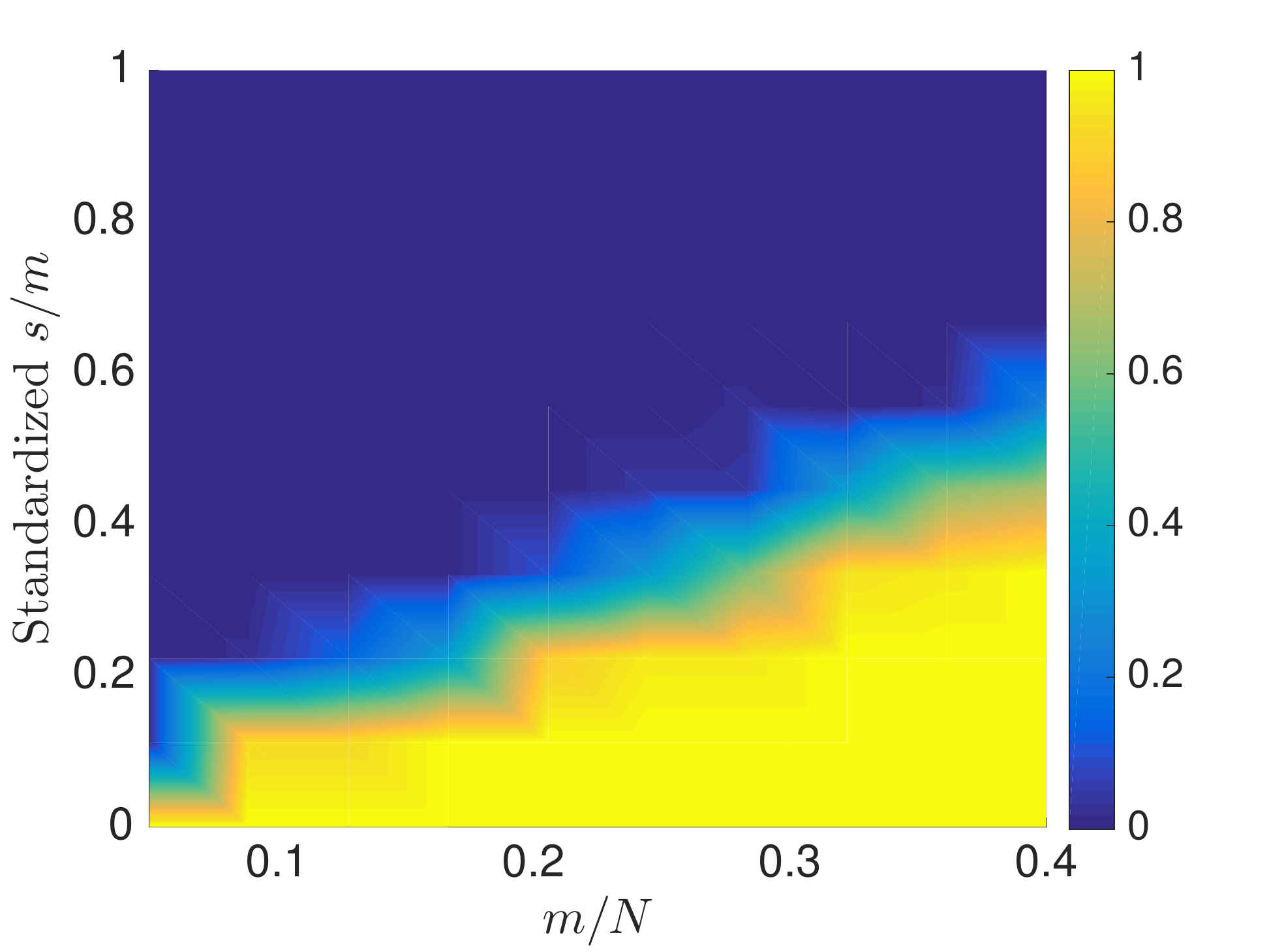} \vspace{-2mm}
\caption{\small Surface plots depicting phase transitions {\em Left panel}:   Weighted $\ell_1$-minimization. {\em Right panel}: unweighted $\ell_1$-minimization.}
\label{fig:ptsurfP}\vspace{-2mm}
\end{figure}
Note that weighted $\ell_1$-minimization has a seemingly high phase transition. This is because the generic sparsities are small for the pre-specified weighted sparsities, which depends on the magnitude of the weights. This can be seen in the left panel of Fig. \ref{fig:k_vs_s_PR} where the largest $k/m$ is roughly $0.63$. Furthermore, we computed the average empirical weighted sparsity in our simulations and as can be seen in the left panel of Fig. \ref{fig:k_vs_s_PR}, these average weighted sparsities are very close to the pre-specified weighted sparsities. Thus confirming our claim that the signals we generate are weighted sparse in expectation.

\begin{figure}[h!]
\centering
\includegraphics[width=0.43\textwidth,height=0.31\textwidth]{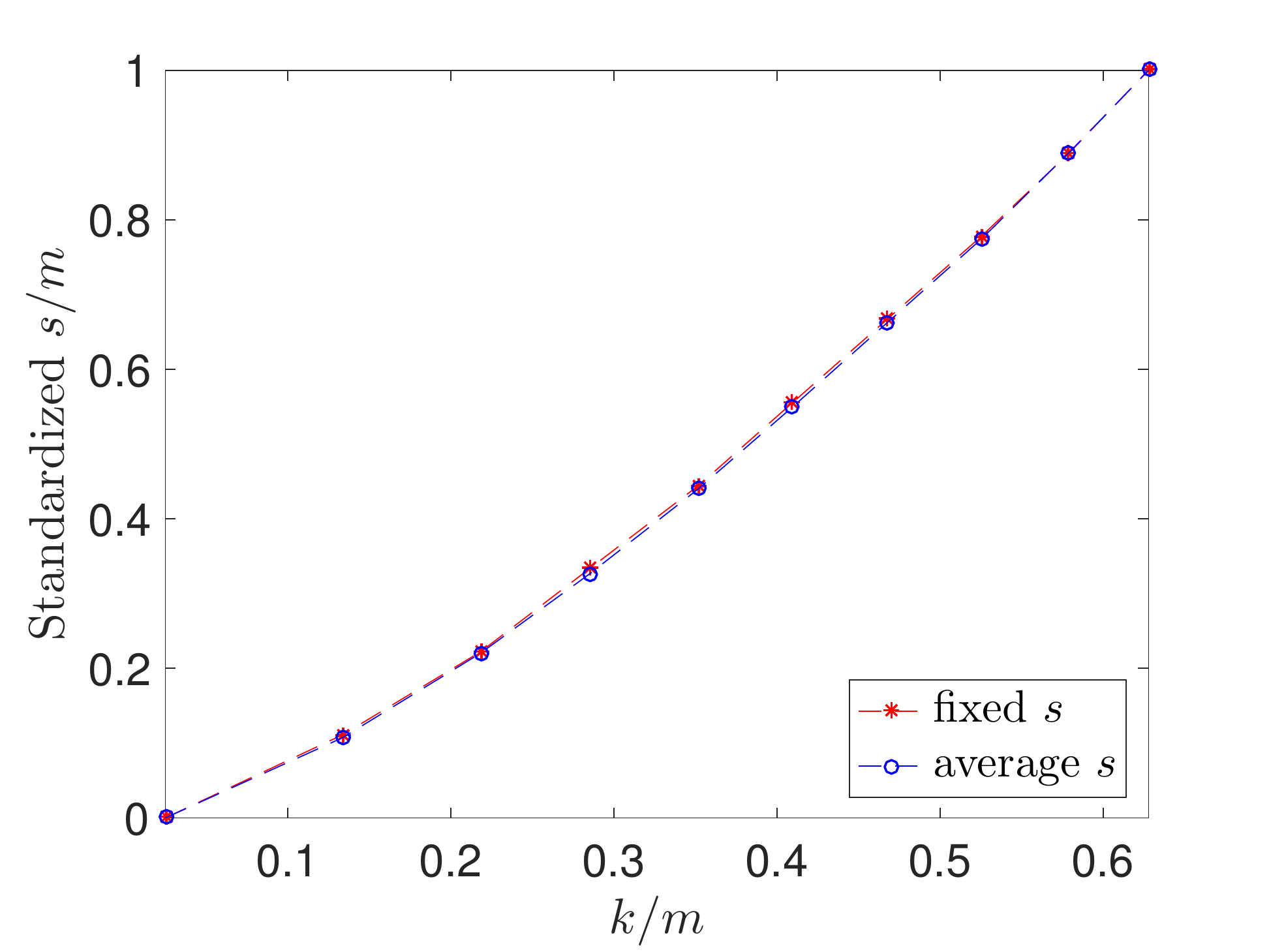} 
\includegraphics[width=0.43\textwidth,height=0.31\textwidth]{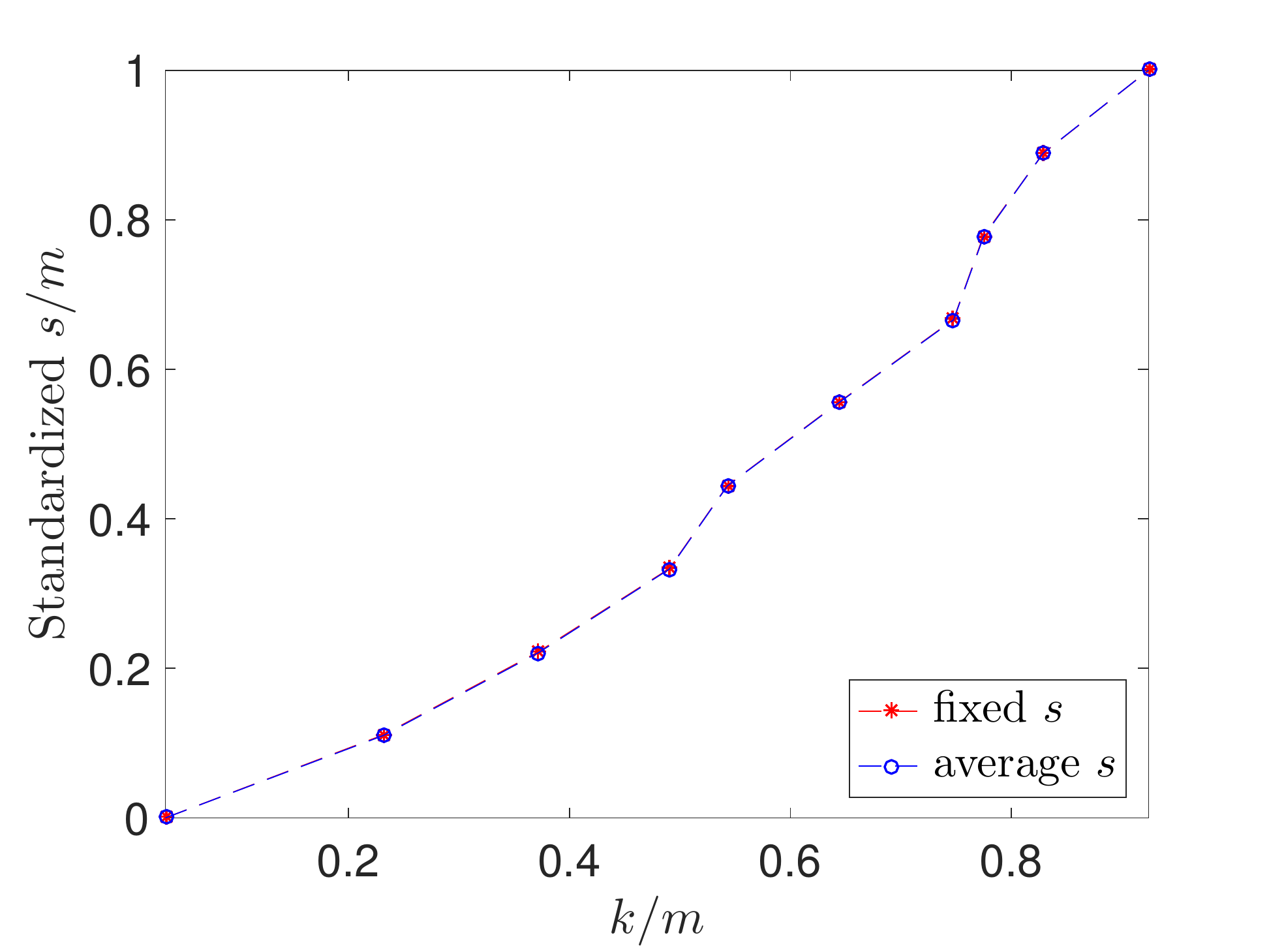} 
\caption{\small Sparsity ratios $k/m$ versus standardized weighted sparsity ratios $s/m$; standardized such that $s/m\in[0,1]$. {\em Left panel}: Polynomially growing weights. {\em Right panel}: Random weights.}
\label{fig:k_vs_s_PR}
\end{figure}

\item {\bf Random weights.}
Precisely, we assign weights that are uniformly randomly generated and normalized such that $\min_j \om_j = 1$. 
Similarly, Fig. \ref{fig:ptsurfR} shows surface plots of empirical probabilities for signals drawn from the distribution of weighted sparse signals with these random weights. Again, by having larger area of yellow, weighted sparse recovery using \eqref{eqn:wl1min}, in the left panel, outperforms generic sparse recovery using \eqref{eqn:l1min}, in the middle panel. 
\begin{figure}[h!]
\centering
\includegraphics[width=0.43\textwidth,height=0.31\textwidth]{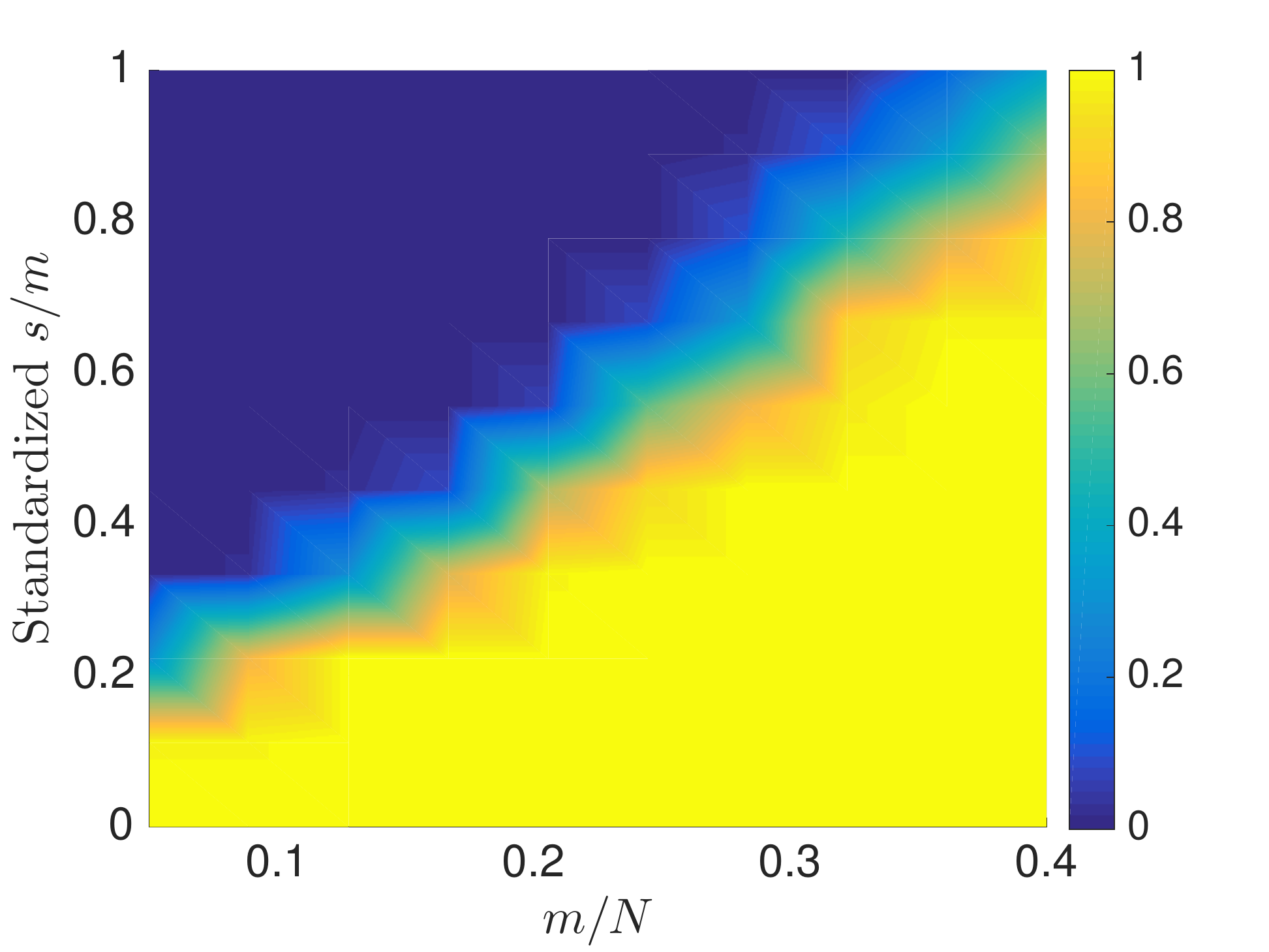} 
\includegraphics[width=0.43\textwidth,height=0.31\textwidth]{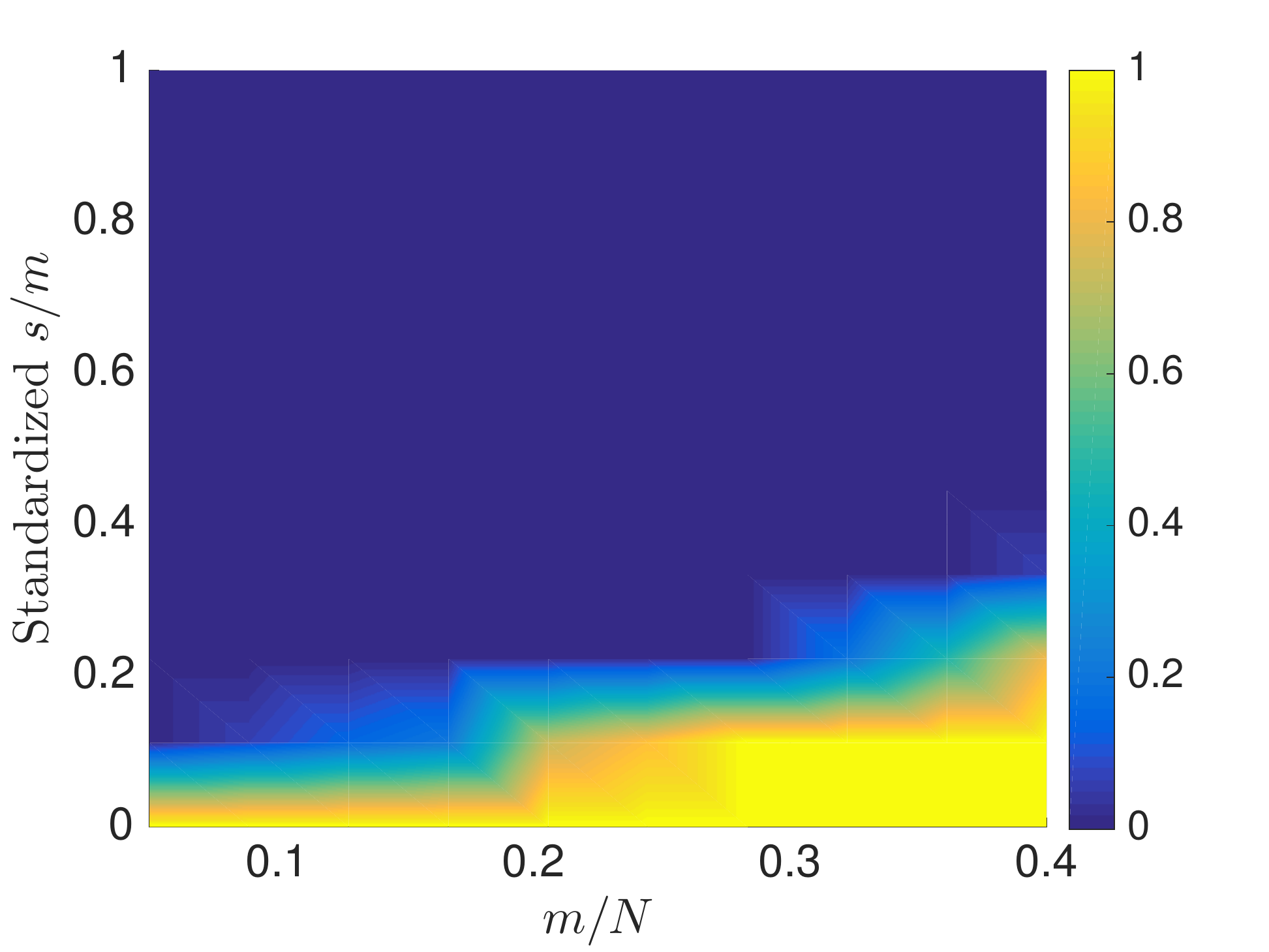} 
\caption{\small Surface plots depicting phase transitions {\em Left panel}:   Weighted $\ell_1$-minimization. {\em Right panel}: unweighted $\ell_1$-minimization.}
\label{fig:ptsurfR}
\end{figure}
\end{itemize}
Note that unweighted $\ell_1$-minimization has a relatively lower phase transition curve for $\ell_1$-minimization with random weights (of relatively smaller weights) than the polynomially growing weights in the right panel of Fig. \ref{fig:ptsurfP}. This is because the generic sparsities are larger for the pre-specified weighted sparsities, see the right panel of Fig. \ref{fig:k_vs_s_PR}. The average weighted sparsities are very close to the pre-specified weighted sparsities, see again the right panel of Fig. \ref{fig:k_vs_s_PR}. Thus again confirming our claim that the signals we generate are weighted sparse in expectation.

\subsection{Sparse recovery with prior support estimates}

In this second set of experiments we show how incorporating additional information about the support of the signal affects sparse recovery using (weighted) $\ell_1$-minimization. The setup is that for a support $\support$ we know an estimate of the support we denote as $\suppest$. We considered support estimates $\suppest$ of different degrees of overlap with the support $\support$. Technically, the overlap is parameterized by $0\leq \alpha,\beta \leq 1$ such that $|\suppest| = \b|\support|$, $|\mathcal{S}\cap\widetilde{\mathcal{S}}| = \alpha |\widetilde{\mathcal{S}}| = \alpha \beta|\mathcal{S}|$; but for simplicity in these experiments we fix $\b = 1$ and vary $\a$ with $\a = \{0,1/3,2/3,1\}$. Recall that $\alpha=0$ means that the estimated support is not aligned with the true support, while $\alpha=1$ indicates perfect alignment.

We assign weights such that $\om_j = 1$ for $j\in\suppest$ and $\om_j = w_2 = 1/w \geq 1$ for $j\in\suppest^c$ with $w = 1-\alpha$ as was prescribed in \cite{mansour2014recovery} as an {\em optimal} choice of weights. The problem dimensions are the same as in the previous section except that we pre-specify $s/m \in \left[\lceil\frac{1}{0.05N}\rceil,\alpha\beta + w_2^2(1-\alpha\beta)\right]$. By the definition of the weighted sparsity $s$ we can compute the generic sparsity $k = \lceil \frac{s}{\alpha\beta + w_2^2(1-\alpha\beta)}\rceil$.

\begin{itemize}
\item {\bf Phase transitions.}
Our theoretical results show that using good weights in the weighted $\ell_1$-minimization reconstruction, we can recover sparse signals on support sets aligned with the weights using fewer measurements $m$. Fig. \ref{fig:ptcontours} shows contour plots of empirical probabilities (the ratio of the number of reconstructions with errors below a given tolerance to the number of realizations) of 50\% and 95\% respectively. The curves depict the sampling rates that give 50\% or more (respectively 95\% or more) chance of getting a ``close enough'' approximation $\sol$ to the true $\x$. 
\begin{figure}[h!]
\centering
\includegraphics[width=0.43\textwidth,height=0.31\textwidth]{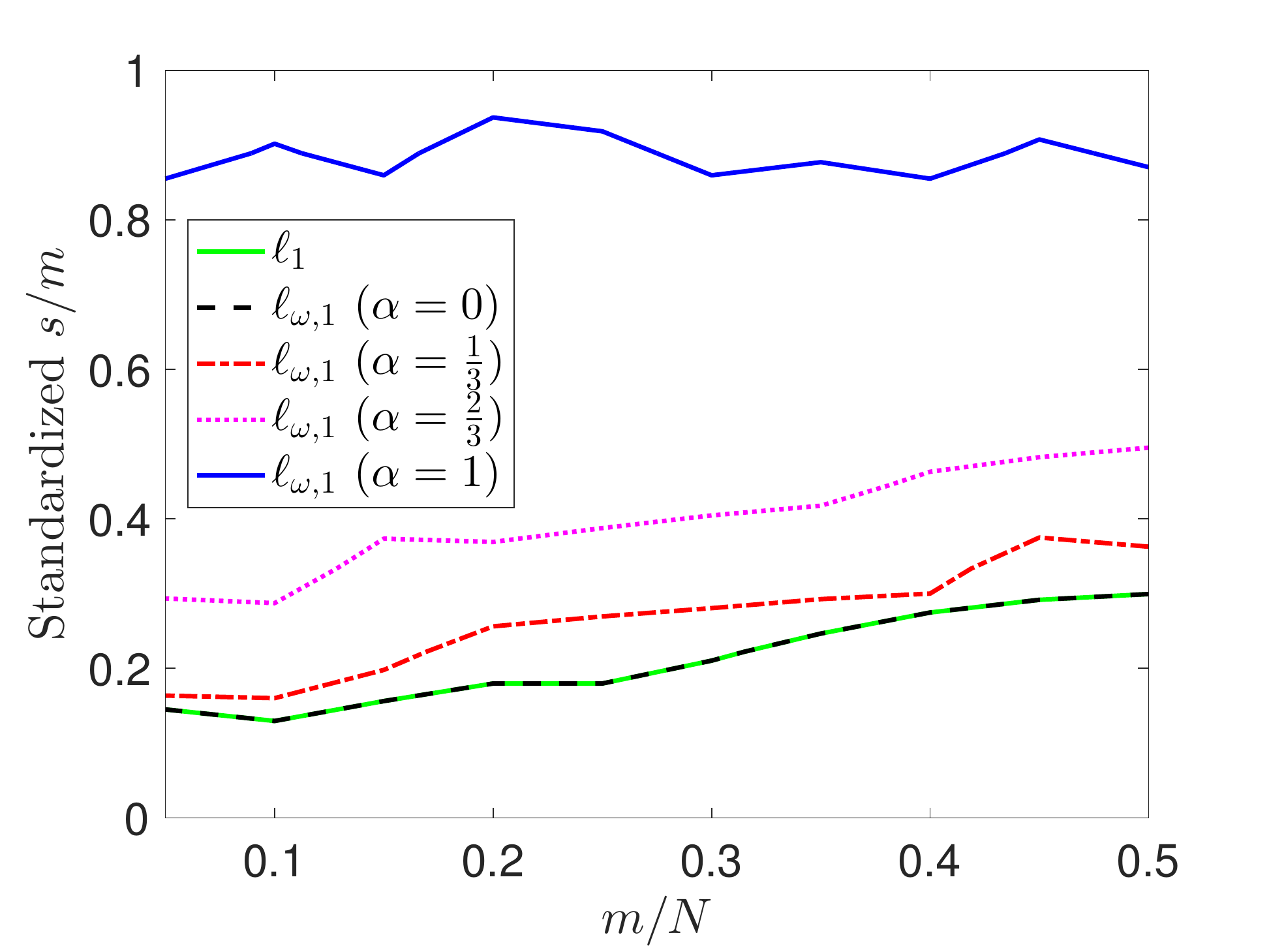}
\includegraphics[width=0.43\textwidth,height=0.31\textwidth]{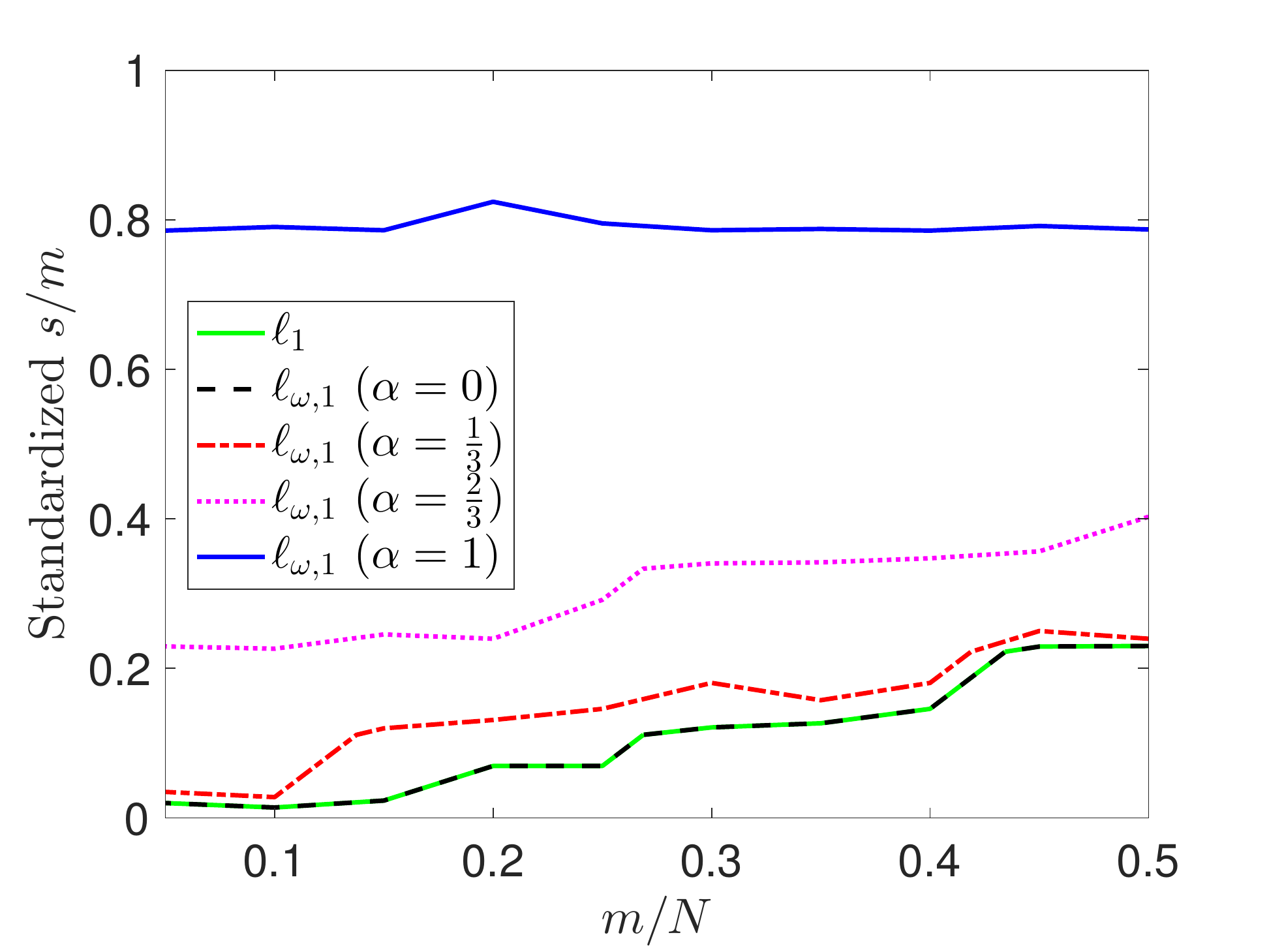} 
\caption{\small Contour plots of phase transitions. {\em Left panel}: 50\% contour curves. {\em Right panel}: 95\% contour lines.}
\label{fig:ptcontours}
\end{figure}

These results show that weighted sparse recovery using \eqref{eqn:wl1min}, outperforms generic sparse recovery using \eqref{eqn:l1min}, if the support estimate is significantly accurate. When $\alpha = 1$ we almost have perfect recovery using weighted $\ell_1$ and even when $\alpha = 2/3$ weighted $\ell_1$ does significantly better than unweighted $\ell_1$. This confirms the results of \cite{mansour2014recovery} in the noiseless setting, which guarantees improvement when $\alpha \geq 50\%$. We see not much improvement of weighted $\ell_1$ over unweighted $\ell_1$ when $\alpha = 1/3$ and when $\alpha = 0$ we get no improvement of weighted $\ell_1$ over unweighted $\ell_1$ as one would expected. 

Furthermore, the left panel of Fig. \ref{fig:kvss_prob} shows how the weighted sparsity varies with generic sparsity. In particular, we see an increase in the weighted sparsity as $\alpha$ increases from $0$ to $1/3$ to $2/3$ but due to the optimal choice of the weights of $w_2^{-1} = 1-\alpha$, $\alpha = 0$ and $\alpha = 1$ have the same weighted sparsity versus generic sparsity profile. The two curves coincide thats why only three curves appeared in the right panel plot. Alternatively, the result depicted in Fig. \ref{fig:ptcontours} can be zoomed into via a slice through the phase transition curves for a fixed $m$, as shown in the right panel of Fig. \ref{fig:kvss_prob} for $~m = N/4$.
\begin{figure}[h!]
\centering
\includegraphics[width=0.43\textwidth,height=0.31\textwidth]{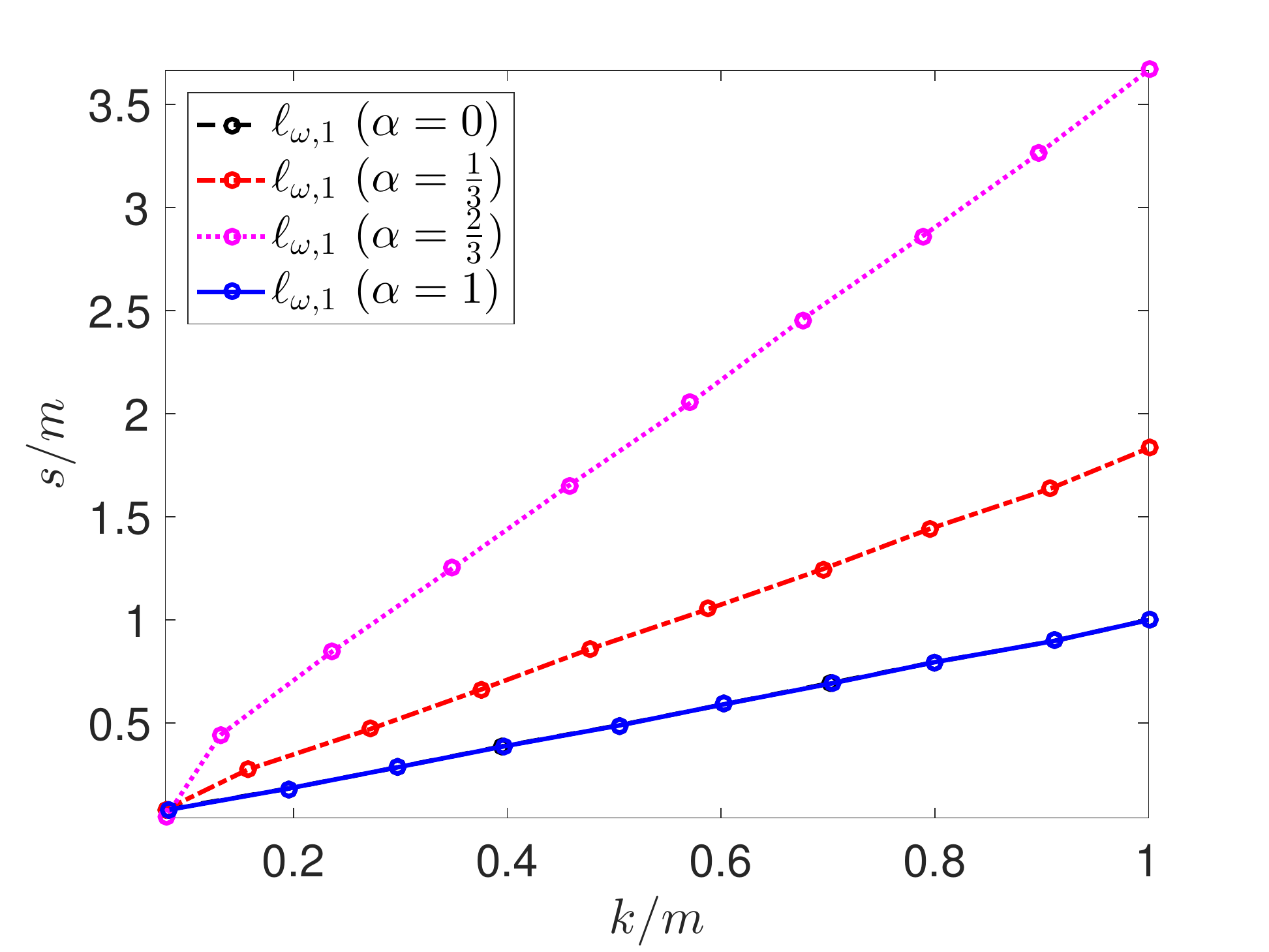} 
\includegraphics[width=0.43\textwidth,height=0.31\textwidth]{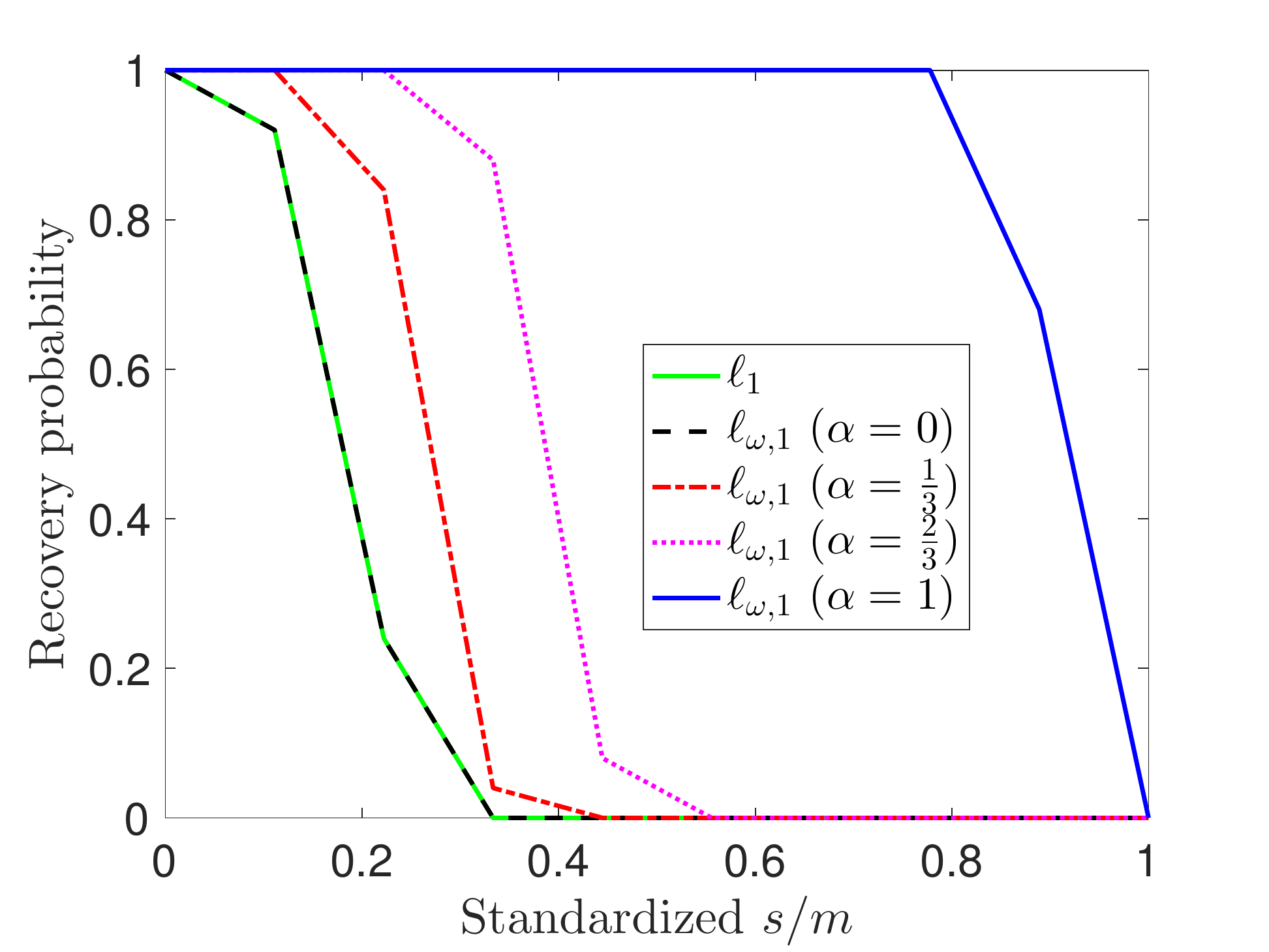} 
\caption{\small {\em Left panel}: Sparsity ratios $k/m$ versus weighted sparsity ratios $s/m$ {\em Right panel}: Recovery probabilities as a function of the weighted sparsity for a fixed $~m=N/4$.}
\label{fig:kvss_prob}
\end{figure}

\item {\bf Recovery errors.}
We also compare recovery errors of the two approaches mentioned above in the two weight setting. We compute the errors for a fixed number of measurements $m = N/4$ and pre-specified $s/m \in \left[\lceil\frac{1}{0.05N}\rceil,\alpha\beta + w_2^2(1-\alpha\beta)\right]$. We denote the recovered signal as $\sol_{alg}$, where {\em alg} is either $\ell_{\om 1}$ or $\ell_1$,  referring to the recovery algorithm used, i.e.  \eqref{eqn:wl1min} or \eqref{eqn:l1min} respectively. We present relative errors for varying $\a \in \{0, 1/3, 2/3, 1\}$.
\begin{figure}[h!]
\centering
\includegraphics[width=0.43\textwidth,height=0.31\textwidth]{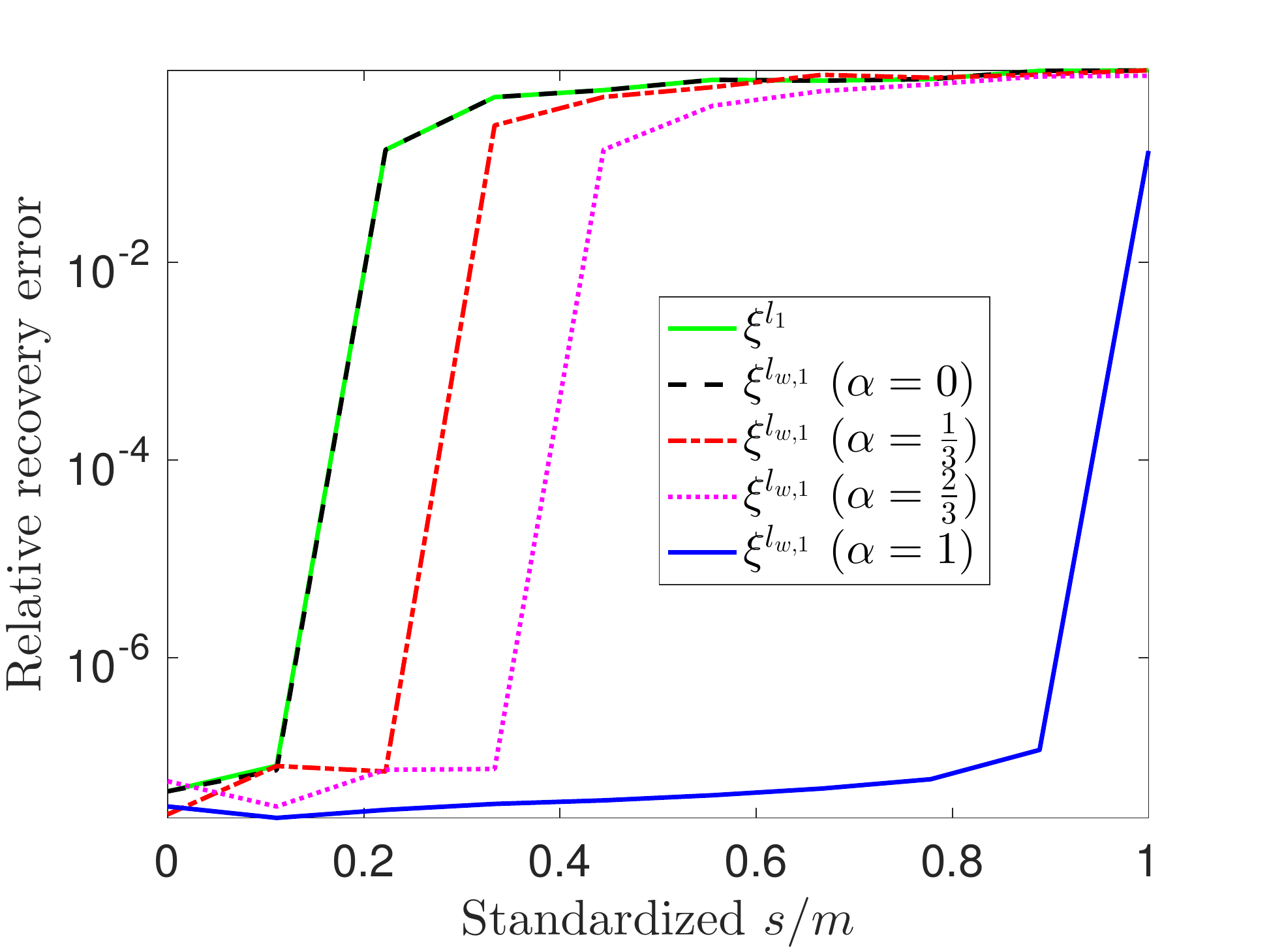} 
\includegraphics[width=0.43\textwidth,height=0.31\textwidth]{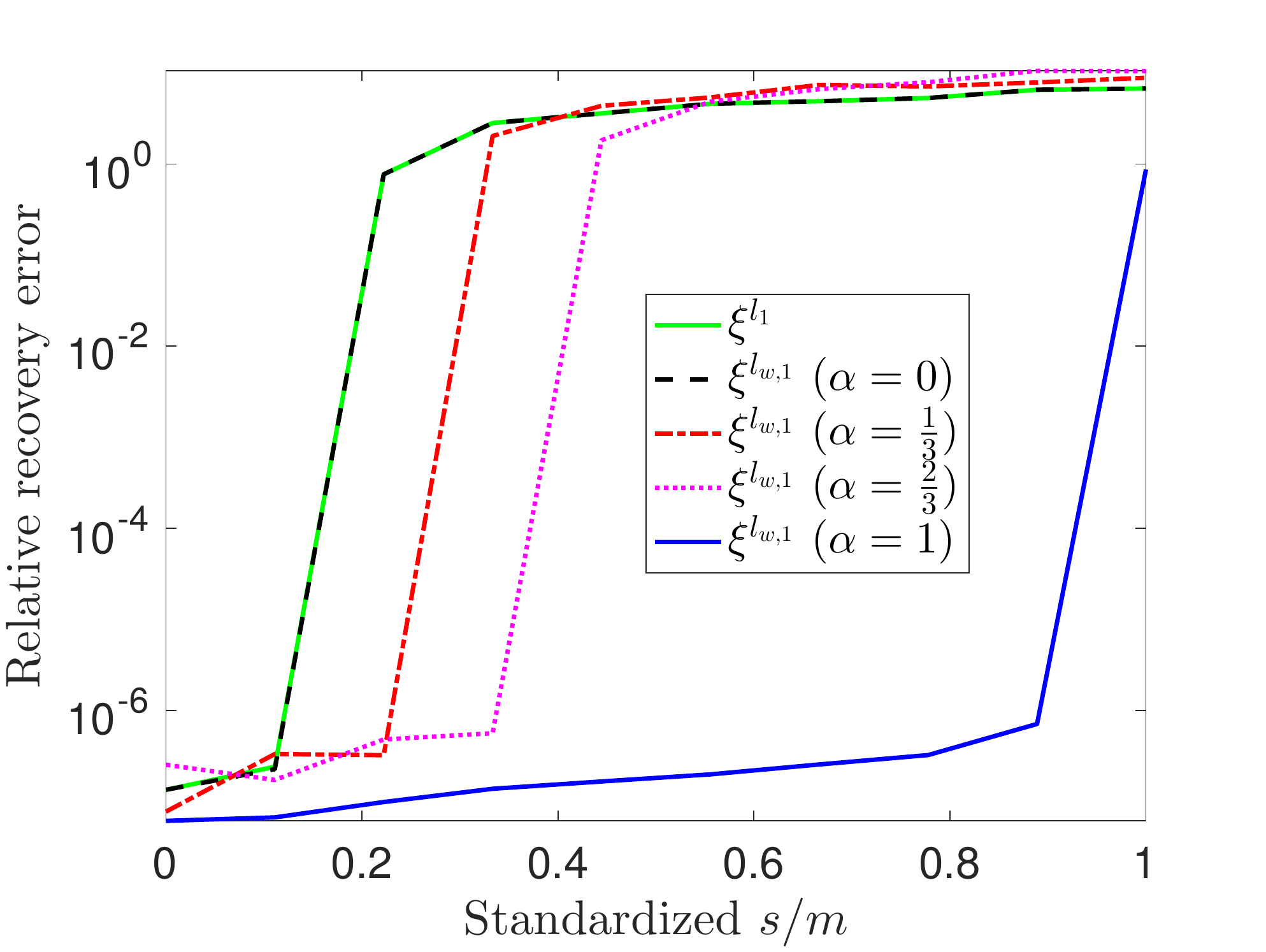} 
\caption{\small Relative reconstruction errors. {\em Left panel}: $\ell_{2}$-norm  {\em Right panel}: $\ell_{\om,1}$-norm.} 
\label{fig:rerrors}
\end{figure}
\end{itemize}
The left panel of Fig. \ref{fig:rerrors} shows the $\ell_{2}$-norm relative errors, i.e. $\xi^{alg} := \frac{\|\sol_{alg} - \x\|_2}{\|\x\|_2}$;  
while the right panel of Fig. \ref{fig:rerrors} shows the $\ell_{\om,1}$-norm relative errors, i.e. $\xi^{alg} := \frac{\|\sol_{alg} - \x\|_{\om,1}}{\|\x\|_2}$.  As implied by the results discussed above, Fig. \ref{fig:rerrors} confirms that sparse recovery using \eqref{eqn:wl1min} get smaller $\ell_2$ (respectively $\ell_{\om,1}$) errors than generic sparse recovery using \eqref{eqn:l1min} as $\alpha$ increases.

\section{Conclusion} \label{sec:concln}
This manuscript unifies, strengthens, and extends previous weighted $\ell_1$ recovery bounds, using the tools of weighted sparsity and weighted null space property.  In particular, we provided a new general sample complexity bound for weighted sparse recovery with Gaussian matrices using weighted $\ell_1$-minimization, by deriving a bound on the Gaussian mean width for nonuniform robust weighted sparse recovery. 

Special cases of our results include the case of polynomially-growing weights and the case of prior information in the form of two weights, whereby we recover the results of \cite{mansour2014recovery}.  Our main sample complexity result unifies and extends these special cases to the general setting of an arbitrary weight distribution $\om_j \geq 1$ by framing the result in terms of weighted sparse recovery bounds, and moreover generalizes these bounds to the noisy setting.  

Several questions remain open for future work.  On the one hand, our results are for Gaussian measurement matrices, and it would be interesting to explore to what extent these results hold for more structured random measurement ensembles.    Finally, it is interesting to see whether our theory can be used to derive theoretical guarantees for re-weighted $\ell_1$-minimization, which is an iterative application of weighted $\ell_1$-minimization where weights are learned in conjunction with the underlying sparse solution.

\section*{Acknowledgements} 
The work was funded, in part, by an NSF CAREER Grant and AFOSR Young Investigator Award.
The authors would like to thank Soledad Villar, for reading the paper and giving valuable feedback that improved the manuscript, and Anastasios Kyrillidis, for providing feedback on the experiments. We would also like to thank the anonymous referees.

\bibliographystyle{myalpha}
\bibliography{bbtex}

\begin{thebibliography}{CRPW12}

\bibitem[AR14]{ayaz2014nonuniform}
U.~Ayaz and H.~Rauhut.
\newblock Nonuniform sparse recovery with subgaussian matrices.
\newblock {\em Electronic Transactions on Numerical Analysis}, 41:167--178,
  2014.

\bibitem[BBRS15]{bouchot2015compressed}
J.-L. Bouchot, B.~Bykowski, H.~Rauhut, and C.~Schwab.
\newblock Compressed sensing {P}etrov-{G}alerkin approximations for parametric
  {PDEs}.
\newblock 2015.

\bibitem[CR06]{candes2006quantitative}
E.~J. Candes and J.~Romberg.
\newblock Quantitative robust uncertainty principles and optimally sparse
  decompositions.
\newblock {\em Foundations of Computational Mathematics}, 6(2):227--254, 2006.

\bibitem[CRPW12]{chandrasekaran2012convex}
V.~Chandrasekaran, B.~Recht, P.~A. Parrilo, and A.~S. Willsky.
\newblock The convex geometry of linear inverse problems.
\newblock {\em Foundations of Computational mathematics}, 12(6):805--849, 2012.

\bibitem[CRT06]{candes2006stable}
E.~J. Cand{\`e}s, J.~Romberg, and T.~Tao.
\newblock Stable signal recovery from incomplete and inaccurate measurements.
\newblock {\em Communications on pure and applied mathematics},
  59(8):1207--1223, 2006.

\bibitem[CWB08]{candes2008enhancing}
E.~J. Cand{\`e}s, M.~B. Wakin, and S.~P. Boyd.
\newblock Enhancing sparsity by reweighted $\ell_1$ minimization.
\newblock {\em Journal of Fourier analysis and applications}, 14(5-6):877--905,
  2008.

\bibitem[Don06]{donoho2006compressed}
D.~L. Donoho.
\newblock Compressed sensing.
\newblock {\em IEEE Trans. Inform. Theory}, 52(4):1289--1306, 2006.

\bibitem[FMSY12]{friedlander2012recovering}
M.~P. Friedlander, H.~Mansour, R.~Saab, and O.~Yilmaz.
\newblock Recovering compressively sampled signals using partial support
  information.
\newblock {\em Information Theory, IEEE Transactions on}, 58(2):1122--1134,
  2012.

\bibitem[FR13]{foucart2013mathematical}
S.~Foucart and H.~Rauhut.
\newblock {\em A mathematical introduction to compressive sensing}.
\newblock Springer, 2013.

\bibitem[Gor88]{gordon1988milman}
Y.~Gordon.
\newblock {\em On Milman's inequality and random subspaces which escape through
  a mesh in ℝ n}.
\newblock Springer, 1988.

\bibitem[Jac10]{jacques2010short}
L.~Jacques.
\newblock A short note on compressed sensing with partially known signal
  support.
\newblock {\em Signal Processing}, 90(12):3308--3312, 2010.

\bibitem[KRZ14]{kabanava2014robust}
M.~Kabanava, H.~Rauhut, and H.~Zhang.
\newblock Robust analysis $\ell_1$-recovery from gaussian measurements and
  total variation minimization.
\newblock {\em arXiv preprint arXiv:1407.7402}, 2014.

\bibitem[KXAH09]{khajehnejad2009weighted}
A.~M. Khajehnejad, W.~Xu, A.~S. Avestimehr, and B.~Hassibi.
\newblock Weighted $\ell_1$ minimization for sparse recovery with prior
  information.
\newblock In {\em Information Theory, 2009. ISIT 2009. IEEE International
  Symposium on}, pages 483--487. IEEE, 2009.

\bibitem[KXAH10]{khajehnejad2010improved}
A.~M. Khajehnejad, W.~Xu, A.~S. Avestimehr, and B.~Hassibi.
\newblock Improved sparse recovery thresholds with two-step reweighted $\ell_1$
  minimization.
\newblock In {\em Information Theory Proceedings (ISIT), 2010 IEEE
  International Symposium on}, pages 1603--1607. IEEE, 2010.

\bibitem[MS14]{mansour2014recovery}
H.~Mansour and R.~Saab.
\newblock Recovery analysis for weighted $\ell_1$-minimization using a null
  space property.
\newblock {\em arXiv preprint arXiv:1412.1565}, 2014.

\bibitem[MY11]{mansour2011weighted}
H.~Mansour and {\"O}.~Yilmaz.
\newblock Weighted-$\ell_1$ minimization with multiple weighting sets.
\newblock In {\em SPIE Optical Engineering $+$ Applications}, pages
  813809--813809. International Society for Optics and Photonics, 2011.

\bibitem[MY12]{mansour2012support}
H.~Mansour and {\"O}.~Yilmaz.
\newblock Support driven reweighted $\ell_1$ minimization.
\newblock In {\em Acoustics, Speech and Signal Processing (ICASSP), 2012 IEEE
  International Conference on}, pages 3309--3312. IEEE, 2012.

\bibitem[OKH12]{oymak2012recovery}
S.~Oymak, M.~A. Khajehnejad, and B.~Hassibi.
\newblock Recovery threshold for optimal weight $\ell_1$ minimization.
\newblock In {\em Information Theory Proceedings (ISIT), 2012 IEEE
  International Symposium on}, pages 2032--2036. IEEE, 2012.

\bibitem[PHD14]{peng2014weighted}
J.~Peng, J.~Hampton, and A.~Doostan.
\newblock A weighted $\ell_1$-minimization approach for sparse polynomial chaos
  expansions.
\newblock {\em Journal of Computational Physics}, 267:92--111, 2014.

\bibitem[RW15]{rauhut2015interpolation}
H.~Rauhut and R.~Ward.
\newblock Interpolation via weighted $\ell_1$ minimization.
\newblock {\em Applied and Computational Harmonic Analysis}, 2015.

\bibitem[Tro14]{tropp2014convex}
J.~A. Tropp.
\newblock Convex recovery of a structured signal from independent random linear
  measurements.
\newblock {\em arXiv preprint arXiv:1405.1102}, 2014.

\bibitem[VBMP07]{von2007compressed}
R.~Von~Borries, J.~C. Miosso, and C.~M. Potes.
\newblock Compressed sensing using prior information.
\newblock In {\em Computational Advances in Multi-Sensor Adaptive Processing,
  2007. CAMPSAP 2007. 2nd IEEE International Workshop on}, pages 121--124.
  IEEE, 2007.

\bibitem[VL10]{vaswani2010modified}
N.~Vaswani and W.~Lu.
\newblock Modified-{CS}: Modifying compressive sensing for problems with
  partially known support.
\newblock {\em Signal Processing, IEEE Transactions on}, 58(9):4595--4607,
  2010.

\bibitem[XKAH10]{xu2010breaking}
W.~Xu, M.~A. Khajehnejad, A.~S. Avestimehr, and B.~Hassibi.
\newblock Breaking through the thresholds: an analysis for iterative reweighted
  $\ell_1$ minimization via the grassmann angle framework.
\newblock In {\em Acoustics Speech and Signal Processing (ICASSP), 2010 IEEE
  International Conference on}, pages 5498--5501. IEEE, 2010.

\end{thebibliography}

\end{document}